\documentclass[final]{amsart}

\usepackage{booktabs}
\usepackage{fullpage}
\usepackage{tristan}
\usepackage{algorithm}
\usepackage{algpseudocode}

\usepackage{pgfplots}
\usepackage{pgfplotstable}
\usepackage{subcaption}
\usepackage{tikz}
\usepackage[giveninits=true,eprint=false,url=false,style=alphabetic]{biblatex}
\newcommand{\trV}{\leb2(\Gamma^-, |\vec \omega \cdot \vec{n}|)}

\renewcommand{\cS}[0]{{\ensuremath{\mathbb{S}}^{2}}}

\newcommand{\vertiii}[1]{{\left\vert\kern-0.25ex\left\vert\kern-0.25ex\left\vert #1 
    \right\vert\kern-0.25ex\right\vert\kern-0.25ex\right\vert}}

\title{Deep Uzawa for Kinetic Transport with Lagrange-Enforced Boundaries}

\author[3]{Charalambos Makridakis}

\author[1]{Aaron Pim}

\author[1,2]{Tristan Pryer}

\author[3]{Nikolaos Rekatsinas}

\address{$^1$ Institute for Mathematical Innovation\\ University of
  Bath, Bath, UK. $^2$ Department of Mathematical Sciences
  \\ University of Bath, Bath, UK. $^3$ Institute of Computational and
  Applied Mathematics, Foundation for Research and Technology (FORTH),
  GR 70013, Crete, Greece}

\begin{document}

\begin{abstract}
  We propose a neural network framework for solving stationary linear
  transport equations with inflow boundary conditions. The method
  represents the solution using a neural network and imposes the
  boundary condition via a Lagrange multiplier, based on a
  saddle-point formulation inspired by the classical Uzawa
  algorithm. The scheme is mesh-free, compatible with automatic
  differentiation and extends naturally to problems with scattering
  and heterogeneous media. We establish convergence of the continuum
  formulation and analyse the effects of quadrature error, neural
  approximation and inexact optimisation in the discrete
  implementation. Numerical experiments show that the method captures
  anisotropic transport, enforces boundary conditions and resolves
  scattering dynamics accurately.
\end{abstract}

\maketitle

\section{Introduction}

\label{sec:introduction}

Many physical systems, from astrophysics to plasma dynamics, require
the resolution of kinetic equations that describe the transport of
particles, energy or radiation through complex media. In particular,
problems in radiative transfer \cite{pomraning2005equations}, neutron
transport \cite{duderstadt1976nuclear}, medical physics
\cite{ashby2025efficient} and rarefied gas dynamics
\cite{bird1994molecular} all involve time-independent transport models
with inflow boundary conditions, where prescribed data enters the
domain along characteristic trajectories. These settings are
especially challenging due to the hyperbolic nature of the problem,
the high dimensionality of the phase space, the anisotropic nature of
the dynamics and the difficulty of enforcing boundary conditions in a
manner that is both stable and accurate, particularly in mesh-free or
data-driven contexts.

A wide array of classical methods has been developed to address these
challenges. Discrete ordinates and spherical harmonics expansions
\cite{sanchez1982review} provide deterministic angular discretisations
for the transport equation, while Monte Carlo approaches offer
stochastic alternatives \cite{cox2022monte}. In the context of
rarefied gas dynamics, foundational work has established the
mathematical structure of kinetic transport through the Boltzmann
equation \cite{cercignani1988boltzmann} and effective particle-based
simulation tools such as direct simulation Monte Carlo
\cite{bird1994molecular} remain widely used for engineering-scale
problems. Deterministic methods, including finite difference, finite
volume, discontinuous Galerkin and spectral discretisations, are also
well developed for transport-type PDEs (see
e.g. \cite{cockburn2004runge,ern2004theory,calloo2025cycle,houston2024efficient})
and have proven effective across a range of settings. Complementing
these, variational approaches based on Lagrange multipliers and
saddle-point formulations offer rigorous frameworks for handling
constraints such as boundary conditions
\cite{babuska1973finite,hinze2008optimization} and continue to
influence the design of modern linear and nonlinear solvers
\cite{benzi2005numerical}.

In recent years, neural network methods have emerged as an alternative
computational paradigm for PDEs, promising better scalability in
high-dimensional problems by avoiding the explicit construction of
meshes \cite{han2018solving,mishra2021physics}. The physics-informed
neural network (PINN) framework \cite{raissi2019physics} incorporates
the PDE residual into a loss functional, allowing neural networks to
learn solutions from data and physical constraints. The flexibility of
this approach has enabled progress on a broad range of forward and
inverse problems, including settings that are otherwise ill-posed or
computationally intractable
\cite{lagaris2000neural,kotary2024learning}. Among these, the Deep
Ritz method \cite{yu2018deep} offers a variational formulation for
elliptic PDEs that connects naturally with classical energy principles
and has inspired several subsequent developments.

Despite their promise, PINNs face challenges when applied to transport
equations, particularly in the treatment of boundary conditions. Since
the loss functional typically enforces inflow data only via soft
penalty terms, a trade-off arises between minimising the PDE residual
and matching boundary values
\cite{krishnapriyan2021characterizing}. For hyperbolic equations where
the boundary strongly dictates the interior solution, this can lead to
bias and reduced accuracy. Several recent works aim to address this
limitation by modifying the weak enforcement strategy. The Deep
Nitsche method \cite{ming2021deep} adapts classical weak imposition of
boundary conditions to the neural setting. Adaptive PINNs
\cite{jagtap2020adaptive}, structured architectures \cite{li2020deep}
and domain decomposition methods offer alternative means of
incorporating boundary information and resolving local features. More
recently, augmented Lagrangian PINNs \cite{basir2023adaptive} and
constraint-satisfying variational methods \cite{zhang2025physics} have
advanced the integration of optimisation theory with neural solvers,
aiming for improved constraint fidelity.

This work introduces a new mesh-free method for stationary kinetic
transport that incorporates inflow boundary conditions via a Lagrange
multiplier. Inspired by the classical Uzawa algorithm
\cite{Uzawa1958}, we adopt a saddle-point iteration in which the
solution and multiplier are represented as neural networks and updated
alternately. This approach, which we term Deep Uzawa, ensures that the
PDE and boundary constraints are treated on equal theoretical
footing. The multiplier is defined as a functional on the inflow
boundary, enforcing the constraint in the trace norm, while the state
network is trained to minimise the Lagrangian, incorporating both the
transport operator and boundary interaction
\cite{makridakis2024deepbcs}. The resulting scheme is compatible with
automatic differentiation and stochastic quadrature, making it
suitable for irregular geometries and data-driven scenarios.

We establish convergence of the continuum Uzawa scheme and analyse the
discrete implementation using neural networks, quantifying the role of
approximation error, quadrature variance and the optimiser itself. Our
analysis shows that the boundary condition is asymptotically enforced
in a strong sense, improving over existing penalty-based PINNs, which
generally guarantee only soft satisfaction. Our formulation also
connects conceptually with constrained neural networks, differentiable
optimisation layers and augmented Lagrangian architectures
\cite{amos2017optnet,amos2017input,son2023enhanced,lu2021physics},
reinforcing the compatibility of classical convex analysis with modern
learning frameworks.

The rest of the paper is organised as
follows. Section~\ref{sec:notation} introduces the PDE model, function
spaces and boundary trace setting. Section~\ref{sec:uzawa} presents
the Uzawa iteration and proves convergence at the continuum
level. Section~\ref{sec:NN} describes the neural architecture and its
discrete implementation. Convergence analysis is given in
Section~\ref{sec:convergence}. Section~\ref{sec:numerics} provides
numerical experiments illustrating the method’s behaviour across a
range of examples. Finally, Section~\ref{sec:conc} offers concluding
remarks and directions for future work.

\section{Notation and Problem Setup}
\label{sec:notation}

We consider the stationary monoenergetic linear transport equation
posed on a spatial domain $D \subset \mathbb{R}^d$, with $d \in
\{2,3\}$. The angular variable $\vec \omega$ ranges over the unit
sphere $\mathbb{S}^{d-1}$, so the phase space is given by the product
set $\mathcal{W} := D \times \mathbb{S}^{d-1}$. The inflow and outflow
portions of the boundary are defined by
\begin{equation}
  \Gamma^\pm := \ensemble{ (\vec{x}, \vec \omega) \in \partial D \times \mathbb{S}^{d-1} }{ \pm \vec{n}(\vec{x}) \cdot \vec \omega > 0 },
\end{equation}
where $\vec{n}(\vec{x})$ denotes the outward unit normal to $\partial
D$ at the point $\vec{x}$.

We seek solutions $u : \mathcal{W} \to \mathbb{R}$ to the boundary
value problem
\begin{equation}
  \vec \omega \cdot \nabla_{\vec{x}} u(\vec{x}, \vec \omega)
  + \sigma_A u(\vec{x}, \vec \omega)
  =
  \sigma_T \left( \frac{1}{|\mathbb{S}^{d-1}|} \int_{\mathbb{S}^{d-1}} \pi(\vec \omega \cdot \vec \omega')\, u(\vec{x}, \vec \omega')\, \mathrm{d}\vec \omega'
  - u(\vec{x}, \vec \omega) \right)
  \quad \text{in } \mathcal{W},
\end{equation}
subject to the inflow boundary condition
\begin{equation}
  u(\vec{x}, \vec \omega) = g(\vec{x}, \vec \omega) \quad \text{for } (\vec{x}, \vec \omega) \in \Gamma^-,
\end{equation}
where $\sigma_A, \sigma_T \geq 0$ are the absorption and scattering
coefficients, and $\pi \in C^1([-1,1]; \mathbb{R}_{\geq 0})$ is a
normalised scattering kernel satisfying
\begin{equation}
  \int_{-1}^1 \pi(y)\, \mathrm{d}y = 1.
\end{equation}

For convenience, we define the transport and scattering operators
acting on $u : \mathcal{W} \to \mathbb{R}$ by
\begin{equation}
  \mathcal{T} u := \vec \omega \cdot \nabla_{\vec{x}} u + \sigma_A u, \qquad
  \mathcal{S} u := \sigma_T \left( u - \frac{1}{|\mathbb{S}^{d-1}|} \int_{\mathbb{S}^{d-1}} \pi(\vec \omega \cdot \vec \omega')\, u(\vec{x}, \vec \omega')\, \mathrm{d}\vec \omega' \right),
\end{equation}
so that the governing equation reads
\begin{equation}
  \mathcal{T} u + \mathcal{S} u = 0 \quad \text{in } \mathcal{W}, \qquad u = g \quad \text{on } \Gamma^-.
\end{equation}

We work in the space
\begin{equation}
  V := \ensemble{ u \in \leb2(\mathcal{W}) }{ \mathcal{T}u + \mathcal{S}u \in \leb2(\mathcal{W}),\ u|_{\Gamma^-} \in \leb2(\Gamma^-) },
\end{equation}
equipped with the graph norm
\begin{equation}
  \Norm{u}_V^2 := \Norm{ (\mathcal{T} + \mathcal{S}) u }_{\leb2(\mathcal{W})}^2 + \Norm{ u }_{\leb2(\Gamma^-)}^2.
\end{equation}
The boundary norm is defined with respect to the transport measure:
\begin{equation}
  \Norm{ u }_{\leb2(\Gamma^-)}^2 := \int_{\Gamma^-} |u(\vec{x}, \vec \omega)|^2\, |\vec{n}(\vec{x}) \cdot \vec \omega|\, \mathrm{d}\vec \omega\, \mathrm{d}\vec{x}.
\end{equation}

\begin{remark}[Continuity of the inflow trace]
The continuity of the trace map $V \to \leb2(\Gamma^-)$ follows from
integration along characteristics. Explicitly, for each characteristic
curve parameterised by arc-length $s$, the inflow trace continuity
follows from the boundedness of the integrals along characteristics
\begin{equation}
  |u(x,\omega)|^2 \leq \int_0^{\tau(x,\omega)} |(T + S)u(x - s\omega, \omega)|^2 \, ds, \quad (x,\omega) \in \Gamma^-,
\end{equation}
with $\tau(x,\omega)$ being the characteristic time-to-boundary. See,
for example, \cite{egger2012mixed} for a rigorous treatment of the
associated graph spaces.
\end{remark}

\begin{remark}[Spectral action of $\mathcal{S}$]
Let $P_n$ denote the Legendre polynomial of degree $n$. If $u(\vec{x},
\cdot)$ is expanded in spherical harmonics, then the scattering
operator $\mathcal{S}$ acts diagonally with eigenvalues
\begin{equation}
  \mu_n := \sigma_T \left( 1 - \int_{-1}^1 \pi(y)\, P_n(y)\, \mathrm{d}y \right).
\end{equation}
This spectral decomposition is useful in understanding the smoothing
effect of angular averaging when $\pi$ is not constant.
\end{remark}

\section{Lagrange Formulation and Uzawa Iteration}
\label{sec:uzawa}

The stationary transport problem posed in Section~\ref{sec:notation}
is boundary-driven and noncoercive. The governing operator
$\mathcal{T} + \mathcal{S}$ is non-self-adjoint and, although the
space $V$ controls the inflow trace in $L^2(\Gamma^-)$, minimising the
residual $\| (\mathcal{T} + \mathcal{S})u \|_{L^2(\mathcal{W})}$ alone
does not ensure that the boundary condition $u = g$ is
satisfied. Penalty-based methods provide one remedy but require tuning
of the weight parameter and do not guarantee strict enforcement.

To enforce the inflow condition consistently while retaining a weak
variational formulation, we introduce a Lagrange multiplier $\lambda
\in \leb2(\Gamma^-)$ and pose a constrained minimisation
problem. Given $g \in \leb2(\Gamma^-)$, we seek $u^* \in V$ satisfying
\begin{equation}
  \mathcal{T} u^* + \mathcal{S} u^* = 0 \quad \text{in } \mathcal{W}, \qquad u^* = g \quad \text{on } \Gamma^-.
\end{equation}
We define the residual-based cost functional
\begin{equation}
  J(u) := \frac{1}{2} \Norm{ \mathcal{T} u + \mathcal{S} u }^2_{\leb2(\mathcal{W})}
  + \frac{\gamma}{2} \Norm{ u - g }^2_{\leb2(\Gamma^-)},
\end{equation}
which penalises mismatch with the inflow condition. The corresponding
Lagrangian is
\begin{equation}
  L(u, \lambda) := J(u) - \ltwop{ \lambda }{ u - g }_{\leb2(\Gamma^-)},
\end{equation}
and the constrained problem becomes: find $(u^*, \lambda^*) \in V \times \leb2(\Gamma^-)$ such that
\begin{equation}
  \label{eq:minimisers}
  (u^*, \lambda^*) := \min_{u \in V} \max_{\lambda \in \leb2(\Gamma^-)} L(u, \lambda).
\end{equation}
This formulation is a saddle-point problem in the spirit of
PDE-constrained optimisation \cite{makridakis2024deep}, adapted to the
transport setting. It avoids strong penalty enforcement and
accommodates noncoercive operators, anisotropy and boundary inflow
dynamics without requiring compactness \cite{makridakis2024deepbcs}.

\begin{definition}[Uzawa iteration]
\label{def:uzawa}
Given an initial multiplier $\lambda^0 \in \leb2(\Gamma^-)$, step size
$\rho > 0$ and iteration index $k \geq 0$, the Uzawa scheme produces
sequences $\{u^k\} \subset V$ and $\{\lambda^k\} \subset
\leb2(\Gamma^-)$ through
\begin{equation}
  \begin{split}
    u^{k} &:= \arg\min_{u \in V}
    L(u, \lambda^k)
    \\
    \lambda^{k+1} &:= \lambda^k + \rho \qp{ u^k - g } \quad \text{in } \leb2(\Gamma^-).
  \end{split}
\end{equation}
\end{definition}

We now turn to the analysis of this scheme. Several auxiliary results
are needed before establishing convergence of the iterates.

\begin{lemma}[Residual identity]\label{lemma:1}
Let $u^k, \lambda^k$ be the iterates of the Uzawa scheme in Definition
\ref{def:uzawa}, and let $(u^*, \lambda^*)$ denote the saddle point
given by (\ref{eq:minimisers}). Then
\begin{equation}
  \Norm{ (\mathcal{T} + \mathcal{S})(u^k - u^*) }^2_{\leb2(\mathcal{W})}
  + \gamma \Norm{ u^k - u^* }^2_{\leb2(\Gamma^-)}
  = \ltwop{ \lambda^k - \lambda^* }{ u^k - u^* }_{\leb2(\Gamma^-)}.
\end{equation}
\end{lemma}

\begin{proof}
By the optimality of $u^*$ and $u^k$ with respect to the Lagrangians
$L(u, \lambda^*)$ and $L(u, \lambda^k)$, we have
\begin{equation}
  \ltwop{ (\mathcal{T} + \mathcal{S}) u^* }{ (\mathcal{T} + \mathcal{S}) \phi }_{\leb2(\mathcal{W})}
  + \gamma \ltwop{ u^* }{ \phi }_{\leb2(\Gamma^-)}
  = \ltwop{ \lambda^* }{ \phi }_{\leb2(\Gamma^-)},
\end{equation}
\begin{equation}
  \ltwop{ (\mathcal{T} + \mathcal{S}) u^k }{ (\mathcal{T} + \mathcal{S}) \phi }_{\leb2(\mathcal{W})}
  + \gamma \ltwop{ u^k }{ \phi }_{\leb2(\Gamma^-)}
  = \ltwop{ \lambda^k }{ \phi }_{\leb2(\Gamma^-)}
\end{equation}
for all $\phi \in V$. Subtracting the two identities and choosing $\phi = u^k - u^*$ yields
\begin{equation}
  \Norm{ (\mathcal{T} + \mathcal{S})(u^k - u^*) }^2_{\leb2(\mathcal{W})}
  + \gamma \Norm{ u^k - u^* }^2_{\leb2(\Gamma^-)}
  = \ltwop{ \lambda^k - \lambda^* }{ u^k - u^* }_{\leb2(\Gamma^-)}. \qedhere
\end{equation}
\end{proof}

\begin{lemma}[Multiplier distance recursion]\label{lemma:2}
Let $u^k, \lambda^k$ be the iterates of the Uzawa scheme given in
Definition \ref{def:uzawa} and let $(u^*, \lambda^*)$ denote the
saddle point given by (\ref{eq:minimisers}). Then
\begin{equation}
  \Norm{ \lambda^{k+1} - \lambda^* }^2_{\leb2(\Gamma^-)}
  =
  \Norm{ \lambda^k - \lambda^* }^2_{\leb2(\Gamma^-)}
  - 2\rho\, \ltwop{ \lambda^k - \lambda^* }{ u^k - u^* }_{\leb2(\Gamma^-)}
  + \rho^2\, \Norm{ u^k - u^* }^2_{\leb2(\Gamma^-)}.
\end{equation}
\end{lemma}

\begin{proof}
From the Uzawa update, we have
\begin{equation}
  \lambda^{k+1} = \lambda^k - \rho (u^k - g),
  \qquad
  \lambda^* = \lambda^* - \rho (u^* - g).
\end{equation}
Subtracting gives
\begin{equation}
  \lambda^{k+1} - \lambda^* = \lambda^k - \lambda^* - \rho (u^k - u^*).
\end{equation}
Taking norms and expanding yields
\begin{align}
  \Norm{ \lambda^{k+1} - \lambda^* }^2_{\leb2(\Gamma^-)}
  &= \Norm{ \lambda^k - \lambda^* - \rho (u^k - u^*) }^2_{\leb2(\Gamma^-)} \\
  &= \Norm{ \lambda^k - \lambda^* }^2_{\leb2(\Gamma^-)}
  - 2\rho\, \ltwop{ \lambda^k - \lambda^* }{ u^k - u^* }_{\leb2(\Gamma^-)}
  + \rho^2\, \Norm{ u^k - u^* }^2_{\leb2(\Gamma^-)}. \qedhere
\end{align}
\end{proof}

\begin{lemma}[Continuous trace map on inflow boundary {\cite[Prop.~2.2]{dahmen2012adaptive}}]
\label{lem:trace}
Let $Y := \overline{ C^1_{\Gamma^+}(D) }^{\Norm{\cdot}_{\leb2(D)}}$
denote the closure in $\leb2(D)$ of the set of continuously
differentiable functions vanishing near $\Gamma^+$. Then there exists
a linear, continuous trace operator
\begin{equation}
  \text{Tr} : Y \to \leb2(\Gamma^-, |\vec \omega \cdot \vec{n}|)
\end{equation}
such that
\begin{equation}
  \Norm{ \text{Tr}(v) }_{\leb2(\Gamma^-)} \leq \sqrt{2}\, \Norm{v}_{\leb2(D)} \qquad \text{for all } v \in Y.
\end{equation}
\end{lemma}

\begin{theorem}[Convergence of the Uzawa iteration]
\label{thm:uzawa-convergence}
If the parameters $\gamma$ and $\rho$ satisfy
\begin{equation}
  0 < \gamma, \qquad 0 < \rho < 2\gamma,
\end{equation}
then the sequence $u^k$ produced by the Uzawa scheme converges to the
solution $u^*$ in the $V$-norm.
\end{theorem}

\begin{proof}
Let $e^k := u^k - u^*$. Combining Lemmas~\ref{lemma:1} and~\ref{lemma:2}, we obtain
\begin{equation}
  \Norm{ \lambda^{k+1} - \lambda^* }^2_{\leb2(\Gamma^-)}
  =
  \Norm{ \lambda^k - \lambda^* }^2_{\leb2(\Gamma^-)}
  - 2\rho\, \Norm{ (\mathcal{T} + \mathcal{S}) e^k }^2_{\leb2(\mathcal{W})}
  - \rho(2\gamma - \rho)\, \Norm{ e^k }^2_{\leb2(\Gamma^-)}.
\end{equation}
Since $0 < \rho < 2\gamma$, the right-hand side is strictly less than
the left unless $e^k = 0$, so the sequence $\Norm{ \lambda^k -
  \lambda^* }^2_{\leb2(\Gamma^-)}$ is strictly decreasing and bounded
below by zero. Hence it converges.

Rearranging and summing over $k = 0$ to $K$ gives a telescoping series:
\begin{equation}
\sum_{k=0}^K \left(
  2\rho\, \Norm{ (\mathcal{T} + \mathcal{S}) e^k }^2_{\leb2(\mathcal{W})}
  + \rho(2\gamma - \rho)\, \Norm{ e^k }^2_{\leb2(\Gamma^-)} \right)
  = \Norm{ \lambda^0 - \lambda^* }^2_{\leb2(\Gamma^-)} - \Norm{ \lambda^{K+1} - \lambda^* }^2_{\leb2(\Gamma^-)}.
\end{equation}
The right-hand side is bounded by $\Norm{ \lambda^0 - \lambda^* }^2_{\leb2(\Gamma^-)}$, so taking the limit $K \to \infty$ yields
\begin{equation}
  \sum_{k=0}^\infty \left(
  2\rho\, \Norm{ (\mathcal{T} + \mathcal{S}) e^k }^2_{\leb2(\mathcal{W})}
  + \rho(2\gamma - \rho)\, \Norm{ e^k }^2_{\leb2(\Gamma^-)} \right) < \infty.
\end{equation}
In particular, each term in the sum tends to zero as $k \to \infty$, so
\begin{equation}
  \Norm{ (\mathcal{T} + \mathcal{S}) e^k }_{\leb2(\mathcal{W})} \to 0,
  \qquad
  \Norm{ e^k }_{\leb2(\Gamma^-)} \to 0.
\end{equation}
By the definition of the $V$-norm,
\begin{equation}
  \Norm{ e^k }_V^2 = \Norm{ (\mathcal{T} + \mathcal{S}) e^k }^2_{\leb2(\mathcal{W})} + \Norm{ e^k }^2_{\leb2(\Gamma^-)},
\end{equation}
we conclude that $\Norm{ e^k }_V \to 0$.
\end{proof}

\begin{remark}[Convergence in the case $\gamma = 0$]
The condition $\gamma > 0$ is not essential for convergence. In the
case $\gamma = 0$, the $V$-norm reduces to $\Norm{ (\mathcal{T} +
  \mathcal{S}) u }_{\leb2(\mathcal{W})}$, and the trace norm $\Norm{ u
}_{\leb2(\Gamma^-)}$ can be controlled using the trace inequality from
Lemma~\ref{lem:trace}. Hence, convergence still holds by the same
argument.
\end{remark}

\begin{theorem}[Strong convergence under dominance of absorption]
  \label{thm:strong-convergence}
  Suppose $\sigma_A, \sigma_T, \rho, \gamma > 0$ satisfy
\begin{equation}\label{eq:absorb_dom}
  \sigma_T < \frac{\sigma_A}{4}, \qquad
  0 < \rho < 2\gamma - \sigma_A - \sqrt{\sigma_A^2 - 16 \sigma_T^2}.
\end{equation}
Then the Uzawa iterates $u^k$ converge strongly to $u^*$ in the norm
\begin{equation}
  \enorm{u}^2 := \Norm{u}^2_{\leb2(\mathcal{W})}
  + \Norm{ \vec \omega \cdot \nabla_{\vec{x}} u }^2_{\leb2(\mathcal{W})}
  + \Norm{u}^2_{\leb2(\Gamma^+)} + \Norm{u}^2_{\leb2(\Gamma^-)}.
\end{equation}
\end{theorem}

\begin{proof}
Let $e^k := u^k - u^*$. From Lemma~\ref{lemma:2}, we have
\begin{equation}\label{eq:residual_identity}
  \Norm{ \lambda^{k+1} - \lambda^* }^2
  =
  \Norm{ \lambda^k - \lambda^* }^2
  - 2\rho \Norm{ (\mathcal{T} + \mathcal{S}) e^k }^2_{\leb2(\mathcal{W})}
  - \rho(2\gamma - \rho)\, \Norm{ e^k }^2_{\leb2(\Gamma^-)}.
\end{equation}

We now derive a coercivity estimate for the residual term. Expanding
yields
\begin{equation}
  \Norm{ (\mathcal{T} + \mathcal{S}) u }^2
  = \Norm{ \mathcal{T} u }^2 + 2 \ltwop{ \mathcal{T} u }{ \mathcal{S} u } + \Norm{ \mathcal{S} u }^2.
\end{equation}
Applying the Cauchy-Schwarz inequality gives
\begin{equation}
  \Norm{ (\mathcal{T} + \mathcal{S}) u }^2
  \geq \Norm{ \mathcal{T} u }^2 - 2 \Norm{ \mathcal{T} u } \Norm{ \mathcal{S} u } + \Norm{ \mathcal{S} u }^2.
\end{equation}
Using the bound $\Norm{ \mathcal{S} u } \leq 2\sigma_T \Norm{u}$, we
apply Young’s inequality with parameter $\beta > 0$
\begin{equation}
  \Norm{ (\mathcal{T} + \mathcal{S}) u }^2
  \geq (1 - 2\sigma_T \beta) \Norm{ \mathcal{T} u }^2 - \frac{2\sigma_T}{\beta} \Norm{u}^2.
\end{equation}
Since $\mathcal{T} u = \vec \omega \cdot \nabla_{\vec{x}} u + \sigma_A
u$, we apply a standard transport identity
\begin{equation}
  \Norm{ \mathcal{T} u }^2
  = \Norm{ \vec \omega \cdot \nabla_{\vec{x}} u }^2
  + \sigma_A \qp{ \Norm{ u }^2_{\leb2(\Gamma^+)} - \Norm{ u }^2_{\leb2(\Gamma^-)} }
  + \sigma_A^2 \Norm{ u }^2.
\end{equation}
Combining these gives
\begin{equation}
  \begin{aligned}
    \Norm{ (\mathcal{T} + \mathcal{S}) u }^2
    &\geq (1 - 2\sigma_T \beta)
    \qp{ \Norm{ \vec \omega \cdot \nabla_{\vec{x}} u }^2
       + \sigma_A \Norm{ u }^2_{\leb2(\Gamma^+)}
       - \sigma_A \Norm{ u }^2_{\leb2(\Gamma^-)} } \\
    &\quad + \left( \sigma_A^2 (1 - 2\sigma_T \beta) - \frac{2\sigma_T}{\beta} \right) \Norm{ u }^2.
  \end{aligned}
\end{equation}
Apply this estimate to $e^k$ and multiply by $2\rho$. Set $\beta =
\frac{1 - \alpha}{2\sigma_T}$ for $\alpha \in (0,1)$, so that $1 -
2\sigma_T \beta = \alpha$. Then
\begin{equation}
  \begin{aligned}
    2\rho \Norm{ (\mathcal{T} + \mathcal{S}) e^k }^2
    &\geq 2\rho \alpha \Norm{ \vec \omega \cdot \nabla_{\vec{x}} e^k }^2
    + 2\rho \sigma_A \alpha \Norm{ e^k }^2_{\leb2(\Gamma^+)} \\
    &\quad + 2\rho \left( \sigma_A^2 \alpha - \frac{4\sigma_T^2}{1 - \alpha} \right) \Norm{ e^k }^2
    - 2\rho \sigma_A \alpha \Norm{ e^k }^2_{\leb2(\Gamma^-)}.
  \end{aligned}
\end{equation}
Substitute this into \eqref{eq:residual_identity}
\begin{equation}
  \begin{aligned}
    \Norm{ \lambda^k - \lambda^* }^2 - \Norm{ \lambda^{k+1} - \lambda^* }^2
    \geq
    &~2\rho \alpha \Norm{ \vec \omega \cdot \nabla_{\vec{x}} e^k }^2
    + 2\rho \sigma_A \alpha \Norm{ e^k }^2_{\leb2(\Gamma^+)} \\
    &+ \left( \rho(2\gamma - \rho - 2\sigma_A \alpha) \right) \Norm{ e^k }^2_{\leb2(\Gamma^-)} \\
    &+ 2\rho \left( \sigma_A^2 \alpha - \frac{4\sigma_T^2}{1 - \alpha} \right) \Norm{ e^k }^2.
  \end{aligned}
\end{equation}
Define
\begin{equation}
  C := \min \{
    2\rho \alpha,\,
    2\rho \sigma_A \alpha,\,
    \rho(2\gamma - \rho - 2\sigma_A \alpha),\,
    2\rho \left( \sigma_A^2 \alpha - \frac{4\sigma_T^2}{1 - \alpha} \right)
  \} > 0,
\end{equation}
which is positive under assumption \eqref{eq:absorb_dom}. Then
\begin{equation}
  C \enorm{e^k}^2 \leq \Norm{ \lambda^k - \lambda^* }^2 - \Norm{ \lambda^{k+1} - \lambda^* }^2.
\end{equation}
Since the right-hand side tends to zero and is summable, we conclude
$\enorm{e^k} \to 0$.
\end{proof}

\begin{remark}[Comparison with graph-norm convergence]
The convergence result of Theorem~\ref{thm:uzawa-convergence}
establishes decay in the graph norm $\Norm{\cdot}_V$, which controls
only the composite residual $(\mathcal{T} + \mathcal{S})u$ and the
inflow trace. In contrast, Theorem~\ref{thm:strong-convergence}
guarantees convergence in a stronger norm. This additional regularity
requires the absorption coefficient to dominate scattering and yields
stronger qualitative properties for the solution.
\end{remark}

\section{Neural Network Approximation and Deep Uzawa Implementation}
\label{sec:NN}
To approximate solutions to the kinetic boundary value problem, we
represent the solution using neural networks as smooth, mesh-free
trial functions. These approximations are embedded within the
Uzawa-type iteration developed in the previous section, allowing for
boundary conditions to be enforced via a Lagrange multiplier
framework. Unlike traditional finite-dimensional discretisations based
on basis expansions, neural networks form a high-capacity function
class that naturally extends to high-dimensional domains and avoids
the need for mesh generation.

In this setting, the neural network output $u_\theta$ serves as the
ansatz at each iteration, while the Uzawa scheme ensures weak
enforcement of the inflow boundary condition. The goal of this section
is to define the neural architecture formally and describe the
discrete algorithm used to jointly update $u_\theta$ and the Lagrange
multiplier $\lambda$.

\subsection{Network Architecture}

We approximate functions $u : \mathcal{W} \to \mathbb{R}$ by neural
networks $u_\theta$, where $\theta$ denotes the collection of weights
and biases across a fixed architecture. These neural networks act as
surrogate functions in the minimisation subproblems of the Uzawa
iteration.

Let $\upsilon := (\vec{x}, \vec \omega) \in \mathcal{W}$ be a generic
point in the physical-angular domain. A fully-connected feedforward
network of depth $L \geq 2$ and width profile $\{d_\ell\}_{\ell =
  0}^L$ (with $d_0 = d$) defines the mapping
\begin{equation}
  u_\theta(\upsilon)
  :=
  \mathcal{A}_L \circ \sigma \circ \mathcal{A}_{L-1} \circ \cdots \circ \sigma \circ \mathcal{A}_1(\upsilon),
\end{equation}
where each layer $\mathcal{A}_\ell : \mathbb{R}^{d_{\ell-1}} \to
\mathbb{R}^{d_\ell}$ is an affine map given by
\begin{equation}
  \mathcal{A}_\ell(y)
  :=
  W_\ell y + b_\ell, \qquad W_\ell \in \mathbb{R}^{d_\ell \times d_{\ell-1}}, \quad b_\ell \in \mathbb{R}^{d_\ell}.
\end{equation}
The nonlinear activation function $\sigma : \mathbb{R} \to \mathbb{R}$
is applied componentwise and is assumed smooth (e.g., $\tanh$, GELU,
or SiLU). The full parameter set is
\begin{equation}
  \theta := \{ W_\ell, b_\ell \}_{\ell=1}^L \in \mathbb{R}^{P},
\end{equation}
where $P = \sum_{\ell=1}^L \left( d_\ell \cdot d_{\ell-1} + d_\ell
\right)$ is the total number of degrees of freedom.

The class of all such neural networks is defined as
\begin{equation}
  \mathcal{N} := \ensemble{ u_\theta : \mathcal{W} \to \mathbb{R} }{ u_\theta \text{ has the form above for some } \theta \in \mathbb{R}^P },
\end{equation}
with corresponding function space
\begin{equation}
  \mathcal{V}_N := \mathcal{N} \subset C^\infty(\mathcal{W}).
\end{equation}
Although $\mathcal{V}_N$ is not linear, the parameter set
\begin{equation}
  \Theta := \ensemble{ \theta \in \mathbb{R}^P }{ u_\theta \in \mathcal{V}_N }
\end{equation}
is a linear subspace of $\mathbb{R}^P$.

When the activation function $\sigma$ is smooth, e.g., $\sigma \in
C^\infty(\mathbb{R})$, the composition structure ensures that
$u_\theta \in C^\infty(\mathcal{W})$ for all $\theta$, with bounded
derivatives on compact subdomains. This regularity has several
important consequences.

First, the directional derivative $\vec \omega \cdot \nabla_{\vec{x}}
u_\theta$ is well-defined and smooth on $\mathcal{W}$, so the
transport operator $\mathcal{T} u_\theta$ is continuous. Second, since
$u_\theta$ is smooth in both $\vec{x}$ and $\vec \omega$, it admits a
classical trace on $\Gamma^-$, and in particular
\begin{equation}
  u_\theta|_{\Gamma^-} \in \trV.
\end{equation}
Together, these properties imply that $u_\theta \in V$ for all
$\theta$, and hence
\begin{equation}
  \mathcal{V}_N \subset V.
\end{equation}

Within this function class, the Uzawa iteration seeks, at each step
$k$, a parameter vector $\theta^k \in \Theta$ that approximately
minimises a discretised Lagrangian $L_h(u_\theta, \lambda^k)$. This
defines a sequence of neural functions $u^k_\theta := u_{\theta^k}$
which serve as the iterates of the scheme. Formally, we define
\begin{equation}
  \theta^k := \arg \min_{\theta \in \Theta} L_h(u_\theta, \lambda^k),
\end{equation}
with the minimisation performed via stochastic gradient descent,
using a fixed number of inner steps per outer iteration. Once
$u^k_\theta$ is computed, the Lagrange multiplier $\lambda^k$, stored
only on a discrete representation of the inflow boundary (e.g.,
$\partial \mathcal{K}_h$), is updated according to the Uzawa
scheme. This process is repeated for a prescribed number of outer
iterations.

\subsection{Quadrature and Discrete Integration}

To evaluate the discrete Lagrangian $L_h(u_\theta, \lambda)$ in
practice, we discretise the integrals over $\mathcal{W}$ and
$\Gamma^-$ using tensor-product quadrature rules in space and
angle. Let $\mathcal{K}_h \subset D$ denote a quadrature grid in the
spatial domain, with associated weights $\{w_y\}_{y \in
  \mathcal{K}_h}$ and let $\mathbb{S}_h \subset \mathbb{S}^{d-1}$
denote a quadrature rule on the unit sphere, with angular weights
$\{w_s\}_{s \in \mathbb{S}_h}$. The tensor-product grid $\mathcal{K}_h
\times \mathbb{S}_h$ is used to approximate integrals over $D \times
\mathbb{S}^{d-1}$.

To approximate integrals over the inflow boundary $\Gamma^-$, we
define a corresponding set of quadrature points $\partial
\mathcal{K}_h \subset \Gamma^-$, with weights $\{w_b\}_{b \in \partial
  \mathcal{K}_h}$ incorporating the angular weight $|\vec{n} \cdot
\vec \omega|$ from the boundary measure.

Various quadrature strategies may be employed. While the notation
above supports a quasi-Monte Carlo viewpoint with structured
low-discrepancy sampling, standard Monte Carlo integration is also
applicable, especially in high dimensions where mesh construction is
prohibitive. Classical tensor-product quadrature such as
Gauss-Legendre rules in space and spherical designs in angle are well
suited to smooth integrands, while sparse grids offer a compromise
between resolution and cost for anisotropic problems. Regardless of
the choice, the discrete formulation of the Lagrangian and its
gradients remains compatible with automatic differentiation.

Given a scalar function $u$, we approximate
\begin{equation}
  \begin{aligned}
    \int_{\mathcal{W}} u(\vec{x}, \vec \omega)\, \d\vec{x}\, \d\vec \omega
    &\approx \sum_{\vec y \in \mathcal{K}_h} \sum_{\vec s \in \mathbb{S}_h} w_{\vec y} w_{\vec s}\, u(\vec y, \vec s), \\
    \int_{\mathbb{S}^{d-1}} u(\vec{x}, \vec \omega)\, \d\vec \omega
    &\approx \sum_{\vec s \in \mathbb{S}_h} w_{\vec s}\, u(\vec{x}, \vec s), \\
    \int_{\Gamma^-} g(\vec{x}, \vec \omega)\, |\vec{n}(\vec{x}) \cdot \vec \omega|\, \d\vec \omega\, \d\vec{x}
    &\approx \sum_{\vec b \in \partial \mathcal{K}_h} w_{\vec b}\, g(\vec b).
  \end{aligned}
\end{equation}
These approximations are applied to all terms in the Lagrangian,
including the residual of the transport-scattering operator and the
inflow boundary mismatch. The resulting discrete energy $L_h(u,
\lambda)$ is minimised at each step of the Uzawa algorithm.

\subsection{Discrete Approximation of the Lagrangian}

With quadrature defined, the discrete Lagrangian for $u : \mathcal{W}
\to \mathbb{R}$ and $\lambda : \Gamma^- \to \mathbb{R}$ is given by
\begin{equation}
  \begin{aligned}
    L_h(u, \lambda) :=
    &~\frac{1}{2} \sum_{\vec y \in \mathcal{K}_h} \sum_{\vec s \in \mathbb{S}_h} w_{\vec y} w_{\vec s}
    \qp{ \vec s \cdot \nabla_{\vec{x}} u(\vec y, \vec s) + (\sigma_A + \sigma_T) u(\vec y, \vec s)
      - \frac{\sigma_T}{|\mathbb{S}^{d-1}|} \sum_{\vec s' \in \mathbb{S}_h} w_{\vec s'}\, \pi(\vec s \cdot \vec s')\, u(\vec y, \vec s') }^2 \\
    &+ \frac{\gamma}{2} \sum_{\vec b \in \partial \mathcal{K}_h} w_{\vec b}\, \qp{ u(\vec b) - g(\vec b) }^2
    - \sum_{\vec b \in \partial \mathcal{K}_h} w_{\vec b}\, \lambda(\vec b)\, \qp{ u(\vec b) - g(\vec b) }.
  \end{aligned}
\end{equation}
This objective defines the target of the inner minimisation in each Uzawa iteration.

\subsection{Implementation of Deep Uzawa}

Algorithm \ref{alg:uzawa} describes the neural Uzawa iteration. At
each outer step, the neural network parameters are updated using
stochastic gradient descent on the discrete Lagrangian and the
Lagrange multiplier is then updated on the inflow boundary quadrature
points.

\begin{algorithm}[h!]
  \caption{Deep Uzawa Iteration}\label{alg:uzawa}
  \begin{algorithmic}[1]
    \Require Initial guess $\lambda^0$, step size $\rho > 0$, number of Uzawa steps $N_{\text{Uz}}$, SGD steps per Uzawa update $N_{\text{SGD}}$, learning rate $\eta$
    \State Initialise network parameters $\theta^0$
    \For{$k = 0$ to $N_{\text{Uz}} - 1$}
      \State $u_\theta^0 \gets$ neural network with parameters $\theta^k$
      \For{$m = 0$ to $N_{\text{SGD}} - 1$}
        \State Compute stochastic gradient $\nabla_\theta L_h(u_\theta^m, \lambda^k)$
        \State $\theta^{m+1} \gets \theta^m - \eta \nabla_\theta L_h(u_\theta^m, \lambda^k)$
      \EndFor
      \State $u^k \gets u_\theta^{N_{\text{SGD}}}$
      \State $\theta^{k+1} \gets \theta^{N_{\text{SGD}}}$
      \State $\lambda^{k+1}(b) \gets \lambda^k(b) - \rho \qp{ u^k(b) - g(b) }$ for all $b \in \partial \mathcal{K}_h$
    \EndFor
  \end{algorithmic}
\end{algorithm}

\section{Convergence of the Neural Network-Based Uzawa Scheme}
\label{sec:convergence}

The goal of this section is to analyse the convergence properties of
the neural iterates produced by Algorithm~\ref{alg:uzawa}. The
continuum theory guarantees that the sequence $(u^k)$, generated by
exact subproblem solves in the space $V$, converges to $(u^*)$. In
practice, however, each update is performed over the neural class
$\mathcal{V}_N \subset V$ and the integrals in $L$ are approximated
using a quadrature. We therefore ask whether the iterates
$(u^k_\theta, \tilde{\lambda}^k)$ produced by the neural Monte Carlo
scheme also converge to the exact solution.

The analysis decomposes the total error into contributions from
approximation, optimisation and discretisation. For fixed
$\tilde{\lambda}^k$, the functional $L(u, \tilde{\lambda}^k)$ is
strongly convex and coercive in $u$ due to the quadratic structure of
the residual and the presence of the stabilisation term. Provided the
network class $\mathcal{V}_N$ is dense in $V$, the neural minimiser
$u^k_\theta$ lies close to the exact solution $u^k$ when quadrature
error and training error are small. At the same time, the mapping
$\lambda \mapsto \arg\min_u L(u, \lambda)$ is continuous under the
$V$-norm, so small changes in $\tilde{\lambda}^k$ induce small changes
in $u^k$. This yields the error bound
\begin{equation}
\Norm{ u^k_\theta - u^* }_V \leq \Norm{ u^k_\theta - u^k }_V + \Norm{ u^k - u^* }_V,
\end{equation}
where the first term reflects approximation and optimisation accuracy
within the neural class and the second term decays due to the
convergence of the continuum Uzawa scheme.

In the dual update, the residuals $u^k_\theta - g$ drive the Lagrange
multiplier sequence. Under gradient-based training, these residuals
tend to zero in $\leb2(\Gamma^-)$ and if the series $\sum_k
\Norm{u^k_\theta - g}^2$ converges, then the iterates
$\tilde{\lambda}^k$ converge to $\lambda^*$. These facts together
imply strong convergence in expectation, i.e., the neural iterates
$u^k_\theta$ converge to $u^*$ in the $V$-norm, and the boundary
residual $\Norm{ u^k_\theta - g }_{\leb2(\Gamma^-)}$ vanishes in
expectation.

\begin{theorem}[Stability of neural iterates under Monte Carlo sampling]
\label{thm:stability_mc}
Let $\mathcal{F}_\theta \subset V$ be a closed subset containing all
neural approximants and fix a dual iterate $\lambda^k \in \trV^*$. Let
$\xi_k$ denote the random quadrature sample used to construct the
discretised Lagrangian $L_h(u) := L_h(u, \lambda^k; \xi_k)$ and let
$L(u) := L(u, \lambda^k)$ denote the exact Lagrangian.

Assume that $L$ is $\mu$-strongly convex on $\mathcal{F}_\theta$ and
that both $L$ and $L_h$ admit $\Lambda$-Lipschitz gradients in the
$V^*$ norm. Suppose $\hat u^k := \arg\min_{\mathcal{F}_\theta} L(u)$
denotes the exact minimiser and that $u^k_\theta \in
\mathcal{F}_\theta$ satisfies the inexact minimisation condition
\begin{equation}
  L_h(u^k_\theta) - \inf_{\mathcal{F}_\theta} L_h(u) \leq \delta_k,
\end{equation}
with $\delta_k \to 0$. Assume further that the quadrature is pointwise
consistent in a neighbourhood of $\hat u^k$, in the sense that
\begin{equation}
  \sup_{u \in B_\delta(\hat u^k)} | L_h(u) - L(u) | \leq \varepsilon_k, \qquad \varepsilon_k \to 0.
\end{equation}
Then there exists $C > 0$ such that
\begin{equation}
  \Norm{ u^k_\theta - \hat u^k }_V \leq C \sqrt{ \delta_k + \varepsilon_k }.
\end{equation}
\end{theorem}

\begin{proof}
  By strong convexity of $L$, the exact minimiser $\hat u^k$ satisfies
  \begin{equation}
    L(v) \geq L(\hat u^k) + \frac{\mu}{2} \Norm{ v - \hat u^k }_V^2 \quad \text{for all } v \in \mathcal{F}_\theta.
  \end{equation}
  This inequality will be used to estimate the error $\Norm{
    u^k_\theta - \hat u^k }_V$.

  We first relate the discrete objective $L_h$ to the exact one. By quadrature consistency,
  \begin{equation}
    | L_h(\hat u^k) - L(\hat u^k) | \leq \varepsilon_k,
    \qquad
    | L_h(u^k_\theta) - L(u^k_\theta) | \leq \varepsilon_k.
  \end{equation}
  The inexact minimisation assumption gives
  \begin{equation}
    L_h(u^k_\theta) \leq \inf_{\mathcal{F}_\theta} L_h + \delta_k \leq L_h(\hat u^k) + \delta_k.
  \end{equation}
  Combining with the first inequality yields
  \begin{equation}
    L_h(u^k_\theta) \leq L(\hat u^k) + \delta_k + \varepsilon_k.
  \end{equation}
  We now use convexity of $L$ at $v = u^k_\theta$:
  \begin{equation}
    L(u^k_\theta) \geq L(\hat u^k) + \frac{\mu}{2} \Norm{ u^k_\theta - \hat u^k }_V^2.
  \end{equation}
  Using the second consistency bound,
  \begin{equation}
    L_h(u^k_\theta) \geq L(u^k_\theta) - \varepsilon_k
    \geq L(\hat u^k) + \frac{\mu}{2} \Norm{ u^k_\theta - \hat u^k }_V^2 - \varepsilon_k.
  \end{equation}
  Comparing this with the previous upper bound on $L_h(u^k_\theta)$ yields
  \begin{equation}
    \delta_k + 2\varepsilon_k \geq \frac{\mu}{2} \Norm{ u^k_\theta - \hat u^k }_V^2,
  \end{equation}
  so that
  \begin{equation}
    \Norm{ u^k_\theta - \hat u^k }_V \leq \sqrt{ \frac{2}{\mu} (\delta_k + 2\varepsilon_k) } \leq C \sqrt{ \delta_k + \varepsilon_k },
  \end{equation}
  with $C := \sqrt{4/\mu}$.
\end{proof}

Building on this stability estimate for the inexact primal minimiser
$u^k_\theta$, we now examine the full neural Uzawa sequence
$(u^k_\theta, \tilde{\lambda}^k)$. In particular, we combine the
approximation, discretisation and optimisation errors with the known
convergence of the continuum scheme to derive convergence results for
the neural iterates in expectation.

\begin{lemma}[Neural approximation in the transport norm]
  \label{lem:NN_approx}
  Let $\mathcal{V}_N$ denote the class of neural networks with smooth
  activation $\sigma \in C^\infty$, finite depth and width and domain
  $\mathcal{W} = D \times \mathbb{S}^{d-1}$. Suppose $u^* \in V \cap
  C^\infty(\overline{\mathcal{W}})$. Then for every $\varepsilon > 0$
  there exists $u_\theta \in \mathcal{V}_N$ such that
  \begin{equation}
    \Norm{ u_\theta - u^* }_V < \varepsilon.
  \end{equation}
\end{lemma}

\begin{proof}[Sketch of proof]
  Since $u^* \in C^\infty(\overline{\mathcal{W}})$, it is smooth in
  both variables with bounded spatial derivatives. Neural networks
  with smooth activation functions (e.g., $\tanh$, GELU) approximate
  such functions arbitrarily well in $\sobh1(\mathcal{W})$ and
  $L^\infty$ norms, provided the depth and width are sufficiently
  large. In particular, approximation results from
  \cite{SchwabZech2019} and \cite{Yarotsky2018} yield $u_\theta \in
  \mathcal{V}_N$ such that
  \begin{equation}
    \Norm{ u_\theta - u^* }_{\sobh1(\mathcal{W})} < \varepsilon.
  \end{equation}
  The $V$-norm depends on $\mathcal{T} u$, $\mathcal{S} u$ and the
  trace $u|_{\Gamma^-}$, all of which are continuous with respect to
  the $\sobh1$ topology when $u$ is smooth. Hence $\Norm{ u_\theta -
    u^* }_V \to 0$ as $\Norm{ u_\theta - u^* }_{\sobh1} \to 0$.
\end{proof}

\begin{lemma}[Monte Carlo consistency of quadrature]
  \label{lem:MC_consistency}
  Let $L(u, \lambda)$ and $L_h(u, \lambda)$ denote the continuous and
  Monte Carlo-discretised Lagrangians, respectively. Assume $u \in V
  \cap C^\infty(\overline{\mathcal{W}})$ and that $\lambda \in \trV^*$
  is represented by a smooth function on $\Gamma^-$. Let $\{(\vec y_i,
  \vec s_i)\}_{i=1}^{N} \subset \mathcal{W}$ be i.i.d.\ samples from
  the uniform measure on $\mathcal{W}$ and let $\{ \vec b_j \}_{j=1}^M
  \subset \Gamma^-$ be i.i.d.\ samples from the inflow-weighted
  measure $|\vec s \cdot \vec n(\vec y)|\, \d\vec y\, \d\vec
  s$. Define
  \begin{equation}
    \begin{aligned}
      L_h(u, \lambda)
      :=
      &~\frac{1}{2N} \sum_{i=1}^N \qp{(\mathcal{T} + \mathcal{S})u(\vec y_i, \vec s_i)}^2 \\
      &+ \frac{\gamma}{2M} \sum_{j=1}^M \qp{ u(\vec b_j) - g(\vec b_j) }^2
      - \frac{1}{M} \sum_{j=1}^M \lambda(\vec b_j) \qp{ u(\vec b_j) - g(\vec b_j) }.
    \end{aligned}
  \end{equation}
  Then, for any $C > 0$ there exists $K > 0$ such that for all $u$
  with $\Norm{u}_V \leq C$,
\begin{equation}
  \mathbb{E} \qp{ \left| L(u, \lambda) - L_h(u, \lambda) \right| } \leq \frac{K}{\sqrt{N}} + \frac{K}{\sqrt{M}}.
\end{equation}
\end{lemma}

\begin{proof}
  Since $u$ and $\lambda$ are smooth and $\Norm{u}_V \leq C$, all
  terms in the integrand of $L(u, \lambda)$ are uniformly
  bounded. Define
  \begin{equation}
    f(\vec y, \vec s) := \qp{ (\mathcal{T} + \mathcal{S}) u(\vec y, \vec s) }^2, \qquad
    r(\vec y, \vec s) := \qp{ u(\vec y, \vec s) - g(\vec y, \vec s) }^2, \qquad
    m(\vec y, \vec s) := \lambda(\vec y, \vec s) \qp{ u(\vec y, \vec s) - g(\vec y, \vec s) }.
  \end{equation}
  All three functions are bounded and continuous on
  $\overline{\mathcal{W}}$, with bounds depending only on $C$, $g$
  and $\lambda$.

  By standard Monte Carlo convergence for bounded integrands with
  i.i.d. samples, we have
  \begin{equation}
    \mathbb{E} \qp{ \left| \int_{\mathcal{W}} f\, \d\vec y\, \d\vec s - \frac{1}{N} \sum_{i=1}^N f(\vec y_i, \vec s_i) \right| } \leq \frac{K_f}{\sqrt{N}},
  \end{equation}
  \begin{equation}
    \mathbb{E} \qp{ \left| \int_{\Gamma^-} r\, |\vec s \cdot \vec n(\vec y)|\, \d\vec y\, \d\vec s - \frac{1}{M} \sum_{j=1}^M r(\vec b_j) \right| } \leq \frac{K_r}{\sqrt{M}},
  \end{equation}
  \begin{equation}
    \mathbb{E} \qp{ \left| \int_{\Gamma^-} m\, |\vec s \cdot \vec n(\vec y)|\, \d\vec y\, \d\vec s - \frac{1}{M} \sum_{j=1}^M m(\vec b_j) \right| } \leq \frac{K_m}{\sqrt{M}}.
  \end{equation}
  Define $K := \max \{ K_f, K_r, K_m \}$. Applying the triangle
  inequality to the three terms in $L$ yields
  \begin{equation}
    \mathbb{E} \qp{ \left| L(u, \lambda) - L_h(u, \lambda) \right| } \leq \frac{K}{\sqrt{N}} + \frac{K}{\sqrt{M}}. \qedhere
  \end{equation}
\end{proof}

\begin{lemma}[Suboptimality decay under gradient descent]
\label{lem:delta_decay}
Let $\tilde{\lambda}^k \in \leb2(\Gamma^-)$ be fixed and let $\xi_k$
denote a fixed Monte Carlo sample. Suppose the map
\begin{equation}
  \theta \mapsto L_h(u_\theta, \tilde{\lambda}^k; \xi_k)
\end{equation}
is differentiable with $L$-Lipschitz gradient in parameter space
\begin{equation}
  \Norm{ \nabla_\theta L_h(u_{\theta_1}, \tilde{\lambda}^k; \xi_k) - \nabla_\theta L_h(u_{\theta_2}, \tilde{\lambda}^k; \xi_k) }
  \leq L \Norm{ \theta_1 - \theta_2 }
  \qquad \text{for all } \theta_1, \theta_2 \in \mathbb{R}^P.
\end{equation}
Let $\theta^{k,0}$ be an initial guess and define a gradient descent
sequence by
\begin{equation}
  \theta^{k,m+1} := \theta^{k,m} - \eta \nabla_\theta L_h(u_{\theta^{k,m}}, \tilde{\lambda}^k; \xi_k), \qquad m = 0, \dots, T_k - 1,
\end{equation}
with step size $0 < \eta \leq \frac{1}{L}$. Let $u^k_\theta :=
u_{\theta^{k,T_k}}$ be the final neural iterate. Then
\begin{equation}
L_h(u^k_\theta, \tilde{\lambda}^k; \xi_k) - \inf_{\theta \in \mathbb{R}^P} L_h(u_\theta, \tilde{\lambda}^k; \xi_k)
\leq \frac{1}{2\eta T_k} \Norm{ \theta^{k,0} - \theta^* }^2,
\end{equation}
where $\theta^* \in \arg\min_{\theta} L_h(u_\theta, \tilde{\lambda}^k;
\xi_k)$. In particular, there exists a constant $C > 0$ such that
\begin{equation}
  \delta^k := \mathbb{E}_{\xi_k} \qp{ L_h(u^k_\theta, \tilde{\lambda}^k; \xi_k) - \inf_{\theta} L_h(u_\theta, \tilde{\lambda}^k; \xi_k) }
\leq \frac{C}{T_k}.
\end{equation}
\end{lemma}

\begin{proof}
  Define $F(\theta) := L_h(u_\theta, \tilde{\lambda}^k; \xi_k)$ and
  let $\theta^* \in \arg\min_\theta F(\theta)$. Since $F$ has an
  $L$-Lipschitz gradient, the classical convergence result for
  gradient descent with step size $\eta \leq 1/L$ gives
\begin{equation}
F(\theta^{k,T_k}) - F(\theta^*) \leq \frac{1}{2\eta T_k} \Norm{ \theta^{k,0} - \theta^* }^2.
\end{equation}
This bound follows from convexity and smoothness; see, for example,
\cite[Theorem 2.1.15]{nesterov2013convex}.

Taking expectation over $\xi_k$ yields
\begin{equation}
\delta^k = \mathbb{E}_{\xi_k} \qp{ F(\theta^{k,T_k}) - F(\theta^*) }
\leq \frac{1}{2\eta T_k} \mathbb{E}_{\xi_k} \Norm{ \theta^{k,0} - \theta^* }^2.
\end{equation}
Define
\begin{equation}
C := \frac{1}{2\eta} \sup_k \mathbb{E}_{\xi_k} \Norm{ \theta^{k,0} - \theta^* }^2,
\end{equation}
which is finite provided the initial guess $\theta^{k,0}$ and the
minimisers $\theta^*$ are uniformly bounded in $k$. This gives the
result.
\end{proof}

\begin{lemma}[Stability of inexact Monte Carlo Uzawa]
\label{lem:stabilityMCUz}
Let $\{u^k_\theta\}_{k \in \mathbb{N}}$ and $\{\tilde{\lambda}^k\}_{k
  \in \mathbb{N}}$ denote the iterates of the inexact Uzawa scheme,
where $u^k_\theta$ is obtained by approximate minimisation of the
Monte Carlo-discretised Lagrangian $L_h(u, \tilde{\lambda}^k;
\xi_k)$. Suppose that
\begin{equation}
  \mathbb{E}_{\xi_k} \qp{ L_h(u^k_\theta, \tilde{\lambda}^k; \xi_k) }
  \leq \inf_{u \in V} \mathbb{E}_{\xi_k} \qp{ L_h(u, \tilde{\lambda}^k; \xi_k) } + \varepsilon^k,
\end{equation}
with $\sum_k \varepsilon^k < \infty$. Assume further that $L_h(\cdot,
\lambda; \xi)$ is coercive in $u$, uniformly in $\lambda$, and that
$\tilde{\lambda}^k \to \lambda^\dagger$ in $\trV^*$. Then,
\begin{equation}
  \mathbb{E}_{\xi_k} \qp{ \Norm{ u^k_\theta - u^* }_V^2 } \to 0 \qquad \text{as } k \to \infty,
\end{equation}
where $u^*$ denotes the solution of the exact constrained problem.
\end{lemma}

\begin{proof}
  Define $u^k := \arg\min_{u \in V} \mathbb{E}_{\xi_k} \qb{L_h(u,
    \tilde{\lambda}^k; \xi_k)}$. By uniform coercivity, there exists
  $\alpha > 0$ such that for all $u \in V$,
\begin{equation}
  \mathbb{E}_{\xi_k} \qb{L_h(u, \tilde{\lambda}^k; \xi_k)}
  \geq
  \mathbb{E}_{\xi_k} \qb{L_h(u^k, \tilde{\lambda}^k; \xi_k)}
  +
  \frac{\alpha}{2} \Norm{ u - u^k }_V^2.
\end{equation}
Applying this with $u = u^k_\theta$ and using the suboptimality bound yields
\begin{equation}
  \frac{\alpha}{2} \mathbb{E}_{\xi_k} \qb{\Norm{ u^k_\theta - u^k }_V^2}
  \leq \varepsilon^k,
\end{equation}
so
\begin{equation}
  \mathbb{E}_{\xi_k} \qb{\Norm{ u^k_\theta - u^k }_V^2}
  \leq \frac{2}{\alpha} \varepsilon^k.
\end{equation}
Next, split the total error:
\begin{equation}
  \mathbb{E}_{\xi_k} \qb{\Norm{ u^k_\theta - u^* }_V^2}
  \leq
  2 \mathbb{E}_{\xi_k} \qb{\Norm{ u^k_\theta - u^k }_V^2}
  +
  2\, \Norm{ u^k - u^* }_V^2.
\end{equation}
The first term tends to zero by summability of $\varepsilon^k$. To
show the second term vanishes, we define the minimiser map
\begin{equation}
  u(\lambda) := \arg\min_{u \in V} \mathbb{E}_{\xi_k} \qb{L_h(u, \lambda; \xi_k)}.
\end{equation}
Then $u^k = u(\tilde{\lambda}^k)$ by definition. Under the stated
assumptions, this map is continuous in the $V$-norm. Since
$\tilde{\lambda}^k \to \lambda^\dagger$, it follows that $u^k \to
u(\lambda^\dagger)$.

Finally, to identify $u(\lambda^\dagger)$ with $u^*$, we observe that
$u^k \to u(\lambda^\dagger)$ in $V$ and $u^k \to u^*$ in
$L^2(\Gamma_-)$ by construction of the Uzawa scheme. Since $u^*$ is
the unique solution of the constrained problem, and both $u^*$ and
$u(\lambda^\dagger)$ solve the same PDE in the limit and satisfy $u =
g$ on $\Gamma_-$, it follows that $u(\lambda^\dagger) = u^*$. Hence
$u^k \to u^*$ and the result follows.
\end{proof}

\begin{remark}[Coercivity and continuity of the Lagrangian map]
  The coercivity of the expected Monte Carlo Lagrangian $\mathbb{E}
  \qp{ L_h(\cdot, \lambda) }$ follows directly from the coercivity of
  the continuous Lagrangian $L(u, \lambda)$ on $V$. If the integrands
  in $L(u, \lambda)$ are bounded and continuous, then the Monte Carlo
  samples yield a uniformly integrable approximation and the
  expectation inherits the same lower bound
\begin{equation}
  \mathbb{E} \qp{ L_h(u, \lambda) } \geq \alpha \Norm{u}_V^2 - C.
\end{equation}
Moreover, the mapping $\lambda \mapsto L_h(u, \lambda)$ is affine in
$\lambda$ for fixed $u$ and hence smooth in the sense of Fr\'echet
differentiability. This structure passes to the expected functional,
so $\lambda \mapsto \mathbb{E} \qp{ L_h(u, \lambda) }$ is Fr\'echet
differentiable for fixed $u$.

As a result, the mapping $\lambda \mapsto \arg\min_u \mathbb{E} \qp{
  L_h(u, \lambda) }$ is continuous under standard conditions from
convex analysis in Hilbert spaces, namely coercivity and Fr\'echet
differentiability of the objective. These properties ensure that the
exact minimisers $u^k$ remain stable as $\tilde{\lambda}^k \to
\lambda^\dagger$ and that convergence $u^k \to u^*$ holds in the strong
$V$-norm.
\end{remark}

\begin{theorem}[Convergence in expectation under Monte Carlo quadrature]
\label{the:convergence}
Let $u^* \in V$ denote the solution to the kinetic boundary value
problem. Let $\{ u^k_\theta \}_{k \in \mathbb{N}}$ and $\{
\tilde{\lambda}^k \}_{k \in \mathbb{N}}$ be the iterates of the Monte
Carlo Uzawa scheme with neural network surrogates. Suppose that at
each iteration $k$, the Lagrangian $L(u, \tilde{\lambda}^k)$ is
approximated by a Monte Carlo estimator $L_h(u, \tilde{\lambda}^k;
\xi_k)$ computed over a mini-batch $\xi_k$ of size $M_k$, with
i.i.d.\ samples from a fixed quadrature distribution on $\mathcal{W}$
and $\Gamma^-$.

Assume all functions are smooth, the integrands in $L$ are uniformly
bounded and Lipschitz continuous in $u$, and that each $u^k_\theta$
satisfies
\begin{equation}
\mathbb{E}_{\xi_k} \qp{ L_h(u^k_\theta, \tilde{\lambda}^k; \xi_k) }
\leq \inf_{u \in \mathcal{V}_N} L(u, \tilde{\lambda}^k) + \varepsilon^k + \eta^k,
\end{equation}
where $\sum_k \varepsilon^k < \infty$ and $\eta^k \leq C/\sqrt{M_k}$
for some constant $C > 0$. Assume $M_k \to \infty$ and that
$\mathcal{V}_N$ is dense in $V$. Then
\begin{equation}
\mathbb{E}_{\xi_k} \qp{ \Norm{ u^k_\theta - u^* }_V^2 } \to 0 \qquad \text{as } k \to \infty.
\end{equation}
\end{theorem}

\begin{proof}
Define $u^k := \arg\min_{u \in V} L(u, \tilde{\lambda}^k)$. By the
triangle inequality,
\begin{equation}
\Norm{ u^k_\theta - u^* }_V \leq \Norm{ u^k_\theta - u^k }_V + \Norm{ u^k - u^* }_V.
\end{equation}
Coercivity of $L(\cdot, \tilde{\lambda}^k)$ implies that for some $\alpha > 0$,
\begin{equation}
L(u, \tilde{\lambda}^k) \geq L(u^k, \tilde{\lambda}^k) + \frac{\alpha}{2} \Norm{ u - u^k }_V^2 \qquad \text{for all } u \in V.
\end{equation}
Taking $u = u^k_\theta$ and rearranging yields
\begin{equation}
\Norm{ u^k_\theta - u^k }_V^2 \leq \frac{2}{\alpha} \qp{ L(u^k_\theta, \tilde{\lambda}^k) - L(u^k, \tilde{\lambda}^k) }.
\end{equation}

The difference on the right-hand side is controlled by the Monte Carlo
error and optimisation error. Since the integrands in $L$ are
Lipschitz and bounded, standard Monte Carlo estimates give
\begin{equation}
\mathbb{E}_{\xi_k} \left| L_h(u, \tilde{\lambda}^k; \xi_k) - L(u, \tilde{\lambda}^k) \right| \leq \eta^k \leq \frac{C}{\sqrt{M_k}}.
\end{equation}
Hence,
\begin{equation}
\mathbb{E}_{\xi_k} \Norm{ u^k_\theta - u^k }_V^2 \leq \frac{2}{\alpha} \qp{ \varepsilon^k + \eta^k }.
\end{equation}

Since $\sum_k \varepsilon^k < \infty$ and $\eta^k \to 0$, the first
term in the triangle inequality tends to zero in expectation.

It remains to show that $\Norm{ u^k - u^* }_V \to 0$. This follows
from the convergence of the exact Uzawa scheme and the continuity of
the minimiser map $\lambda \mapsto \arg\min_u L(u, \lambda)$ under
coercivity and smooth dependence on $\lambda$. Since
$\tilde{\lambda}^k \to \lambda^*$ and $u^* = \arg\min_u L(u,
\lambda^*)$, we conclude $u^k \to u^*$ in $V$.

Combining both terms, we obtain
\begin{equation}
\mathbb{E}_{\xi_k} \qp{ \Norm{ u^k_\theta - u^* }_V^2 } \to 0. \qedhere
\end{equation}
\end{proof}

\begin{corollary}[Convergence of boundary residual in expectation]
\label{cor:boundary}
Under the assumptions of Theorem~\ref{the:convergence}, the boundary
residual vanishes in expectation:
\begin{equation}
\mathbb{E} \qp{ \Norm{ u^k_\theta - g }^2_{\leb2(\Gamma^-)} } \to 0 \qquad \text{as } k \to \infty.
\end{equation}
\end{corollary}

\begin{proof}
By Theorem~\ref{the:convergence}, we have $\mathbb{E} \qp{ \Norm{
    u^k_\theta - u^* }_V^2 } \to 0$. Since the trace map $V \to
\leb2(\Gamma^-)$ is continuous, it follows that
\begin{equation}
\mathbb{E} \qp{ \Norm{ u^k_\theta - u^* }^2_{\leb2(\Gamma^-)} } \to 0.
\end{equation}
Using $u^* = g$ on $\Gamma^-$, we obtain
\begin{equation}
\mathbb{E} \qp{ \Norm{ u^k_\theta - g }^2_{\leb2(\Gamma^-)} }
= \mathbb{E} \qp{ \Norm{ u^k_\theta - u^* }^2_{\leb2(\Gamma^-)} },
\end{equation}
which tends to zero.
\end{proof}

\begin{remark}[Interpretation of convergence results]
Theorem~\ref{the:convergence} and Corollary~\ref{cor:boundary} show
that, under idealised assumptions on network expressivity, optimiser
accuracy and quadrature consistency, the neural iterates $u^k_\theta$
converge in expectation to $u^*$ in the $V$-norm and satisfy the
inflow condition in $\leb2(\Gamma^-)$ asymptotically.

These results hold in expectation, not for individual realisations. In
particular, they do not imply that $u^k_\theta = g$ pointwise at
finite $k$, but rather that boundary residuals vanish on average as $k
\to \infty$. This contrasts with penalty-based PINNs, where boundary
errors typically persist even in the limit.
\end{remark}

\section{Numerical Experiments}
\label{sec:numerics}

We now present a numerical experiment illustrating the performance of
the Deep Uzawa algorithm for a simple transport problem. The spatial
domain is $D = [0,1]^2 \subset \mathbb{R}^2$ and the angular domain is
the unit circle $S = \mathbb{S}^1$. The inflow boundary is defined as
\begin{equation}
\Gamma^- := \ensemble{ (x, \vec{\omega}) \in \partial D \times S }{ \vec{\omega} \cdot \vec{n}(x) < 0 },
\end{equation}
where $\vec{n}(x)$ denotes the outward unit normal to $\partial D$. We
parameterise $\vec{\omega}$ as $\vec{\omega} = (\cos \theta, \sin
\theta)$ with $\theta \in [0, 2\pi]$.

\subsection*{Example 1: Directional Transport with Isotropic Absorption}

We solve the steady linear transport equation
\begin{equation}
  \vec{\omega} \cdot \nabla u(\vec{x}, \vec{\omega}) + \sigma(\vec{x}, \vec{\omega}) u(\vec{x}, \vec{\omega}) = f(\vec{x}, \vec{\omega}), \qquad (\vec{x}, \vec{\omega}) \in D \times S,
\end{equation}
subject to the inflow boundary condition $u = g$ on $\Gamma^-$. In
this example, we take $\sigma(\vec{x}, \vec{\omega}) \equiv 1$, $g
\equiv 0$, and define the source term $f$ by
\begin{equation}
  f(\vec{x}, \vec{\omega}) = \begin{cases}
    1, & \text{if } \Norm{ \vec{x} - (0.5, 0.5) } \leq r_0, \\
    0, & \text{otherwise},
  \end{cases}
\end{equation}
with $r_0 = 1$.

 We train a single neural network $u_\theta$ to approximate the
 solution simultaneously across all directions $\vec{\omega}$. In
 post-processing, we evaluate $u_\theta(\vec{x}, \vec{\omega})$ for
 fixed values of $\theta \in [0, 2\pi]$, where $\vec{\omega} = (\cos
 \theta, \sin \theta)$, in order to visualise angular slices of the
 solution. We also compute the scalar flux
\begin{equation}
  \phi(\vec{x}) := \int_{\mathbb{S}^1} u_\theta(\vec{x}, \vec{\omega})\, \d\vec{\omega},
\end{equation}
which gives a directionally averaged measure of particle density. The
results, shown in Figure~\ref{fig:example1:summary}, confirm that the
method captures the directionality of the transport process, correctly
propagates the influence of the source along characteristics, and
produces physically meaningful scalar flux profiles.

The plot of $u_\theta$ on the inflow boundary verifies that the
network satisfies the condition $u = g = 0$ approximately, while the
training loss decays steadily, in agreement with the convergence
theory of Section~\ref{sec:convergence}.

\begin{figure}[h!]
  \centering
  \begin{subfigure}[t]{0.48\textwidth}
    \centering
    \includegraphics[width=\linewidth]{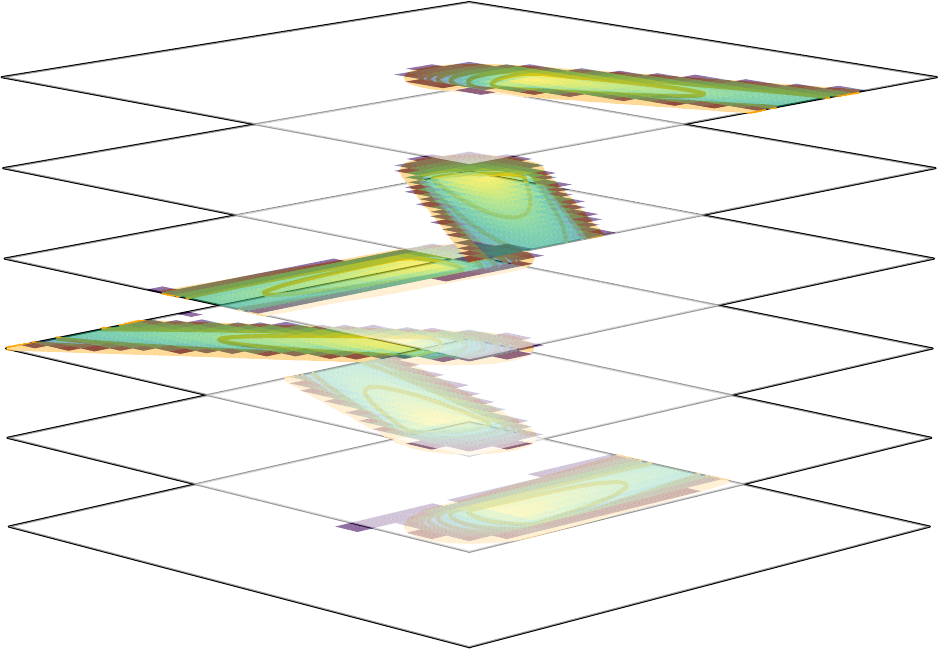}
    \caption{Directional slices of $u_\theta(\vec{x}, \vec{\omega})$ for selected values of $\vec{\omega}$.}
    \label{fig:example1:landscape}
  \end{subfigure}
  \hfill
  \begin{subfigure}[t]{0.48\textwidth}
    \centering
    \includegraphics[width=\linewidth]{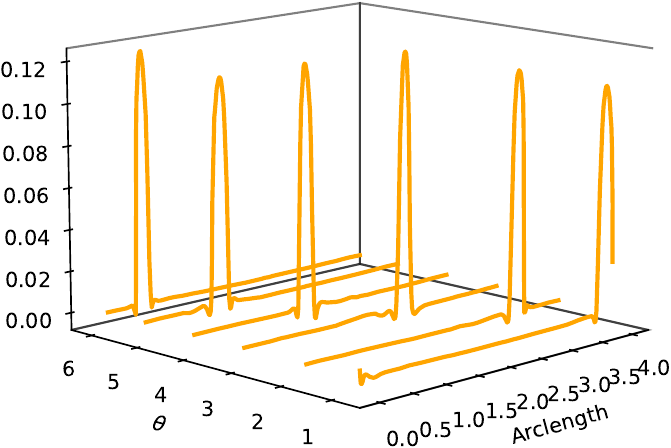}
    \caption{Solution $u_\theta$ on the inflow boundary $\Gamma^-$.}
    \label{fig:example1:boundary}
  \end{subfigure}

  \medskip
  \begin{subfigure}[t]{0.48\textwidth}
    \centering
    \includegraphics[width=\linewidth]{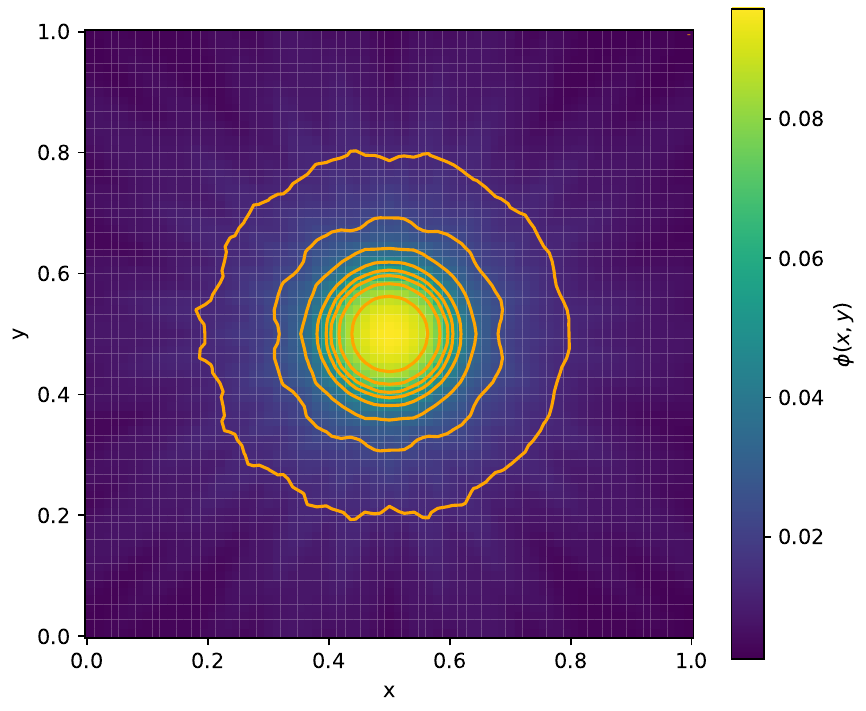}
    \caption{Scalar flux $\phi(\vec{x})$ with contour overlay.}
    \label{fig:example1:scalar}
  \end{subfigure}
  \hfill
  \begin{subfigure}[t]{0.48\textwidth}
    \centering
    \includegraphics[width=\linewidth]{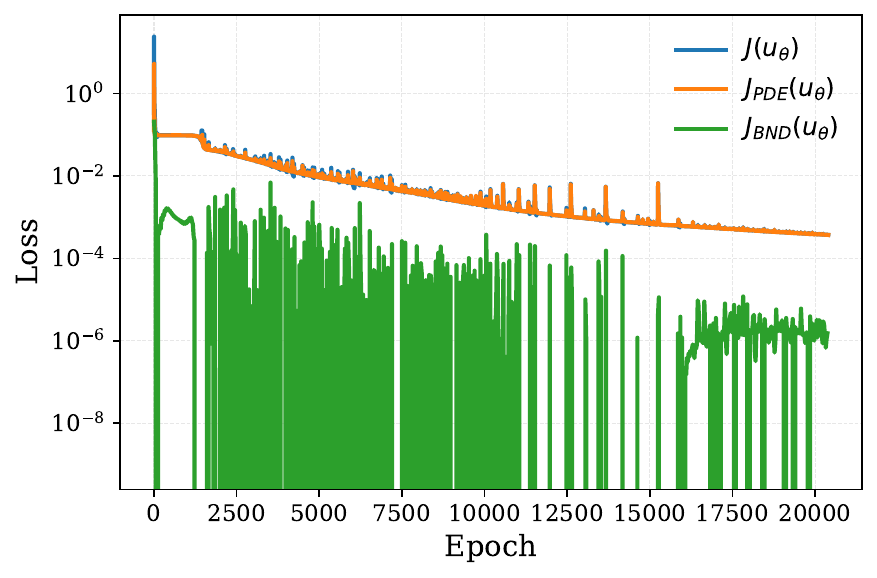}
    \caption{Training loss history.}
    \label{fig:example1:loss}
  \end{subfigure}

  \caption{Summary of results for Example 1. The learned solution
    exhibits directional propagation consistent with transport
    dynamics, satisfies the inflow condition approximately, and
    produces a physically meaningful scalar flux.}
  \label{fig:example1:summary}
\end{figure}

\subsection*{Example 2: Shadowing from a Local Obstacle}

We test the method’s ability to capture directional transport and
shadow formation. Let $D = [0,1]^2$ and $S = \mathbb{S}^1$. We set
$\sigma(\vec{x}) = 1$ in most of the domain, but define a
high-absorption obstacle at the centre:
\begin{equation}
  \sigma(\vec{x}) = 
  \begin{cases}
    50, & \text{if } \Norm{\vec{x} - (0.5, 0.5)} \leq 0.15, \\
    1, & \text{otherwise}.
  \end{cases}
\end{equation}
We take $f \equiv 0$, and prescribe boundary data to model a leftward-pointing beam:
\begin{equation}
  g(\vec{x}, \vec{\omega}) =
  \begin{cases}
    1, & \vec{x}_1 = 0 \text{ and } \vec{\omega} = (1, 0), \\
    0, & \text{otherwise}.
  \end{cases}
\end{equation}
This setup generates a sharp directional front which should cast a
visible shadow behind the obstacle.

Figure~\ref{fig:example3:summary} shows the learned solution. The
scalar flux confirms that the beam is partially blocked and attenuated
in the shadow region, while the inflow boundary condition is
approximately satisfied. The loss plot shows steady optimisation, with
both the PDE and boundary components contributing to the residual.

\begin{figure}[h!]
  \centering
  \begin{subfigure}[t]{0.48\textwidth}
    \centering
    \includegraphics[width=\linewidth]{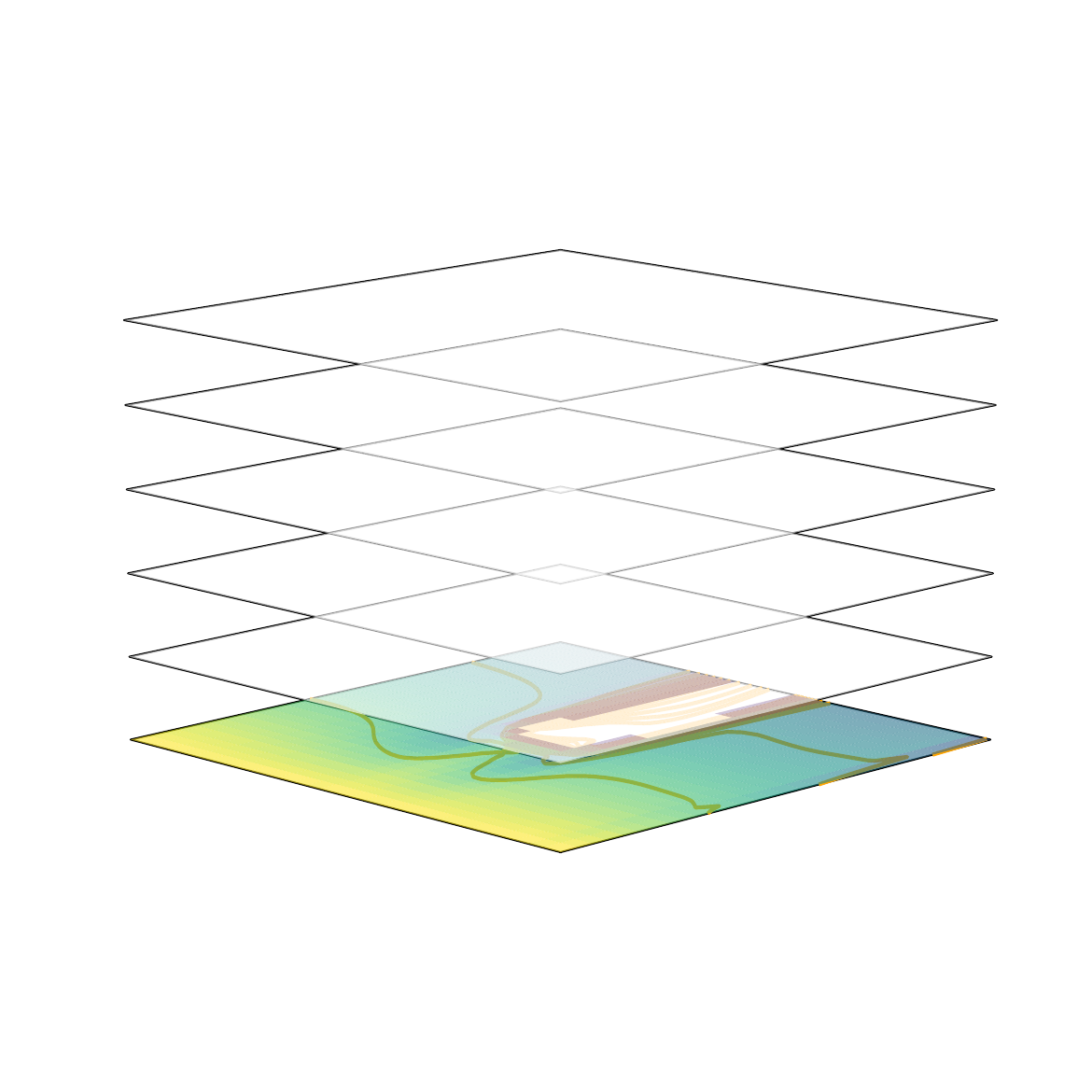}
    \caption{Directional slices of $u_\theta(\vec{x}, \vec{\omega})$ for selected $\vec{\omega}$.}
    \label{fig:example3:landscape}
  \end{subfigure}
  \hfill
  \begin{subfigure}[t]{0.48\textwidth}
    \centering
    \includegraphics[width=\linewidth]{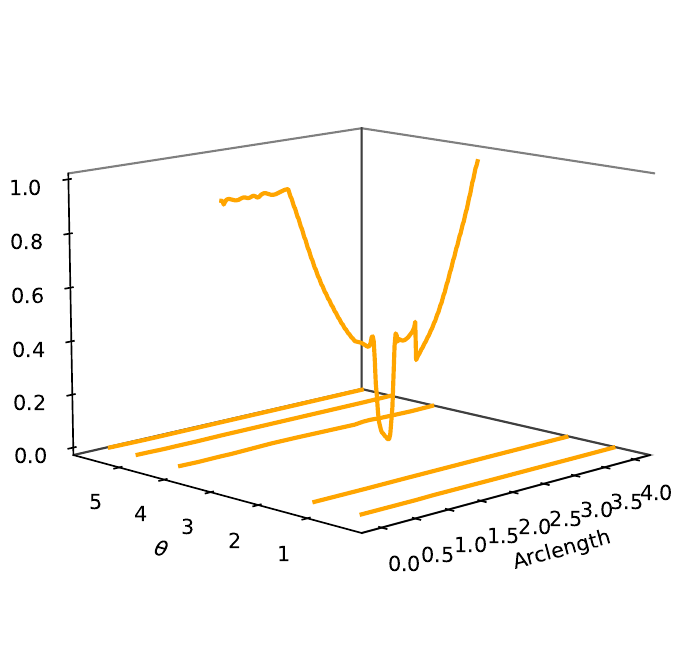}
    \caption{Solution $u_\theta$ on the inflow boundary $\Gamma^-$.}
    \label{fig:example3:boundary}
  \end{subfigure}
  
  \medskip

  \begin{subfigure}[t]{0.48\textwidth}
    \centering
    \includegraphics[width=\linewidth]{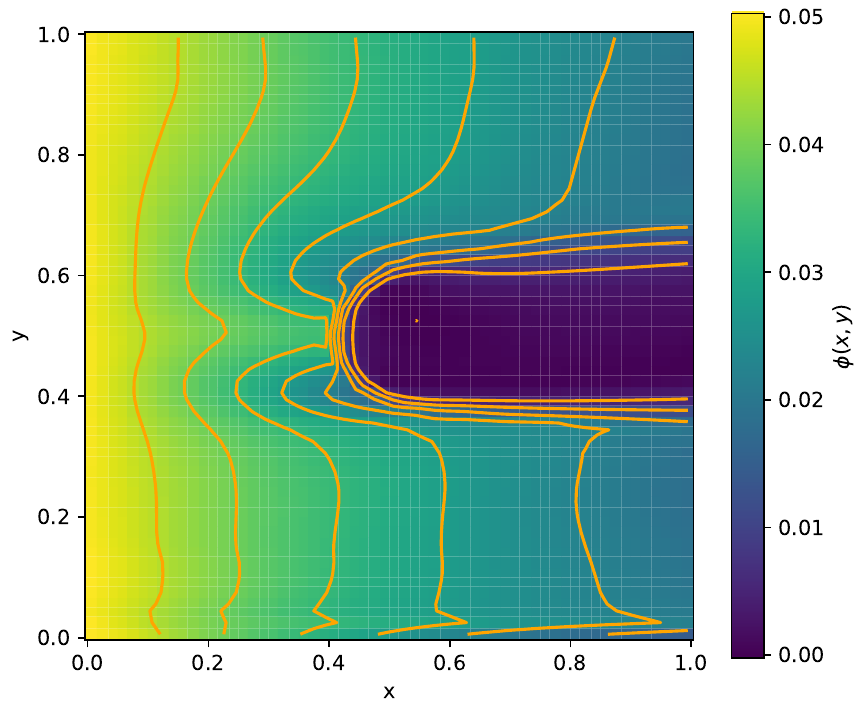}
    \caption{Scalar flux $\phi(\vec{x})$ showing shadow behind the obstacle.}
    \label{fig:example3:scalar}
  \end{subfigure}
  \hfill
  \begin{subfigure}[t]{0.48\textwidth}
    \centering
    \includegraphics[width=\linewidth]{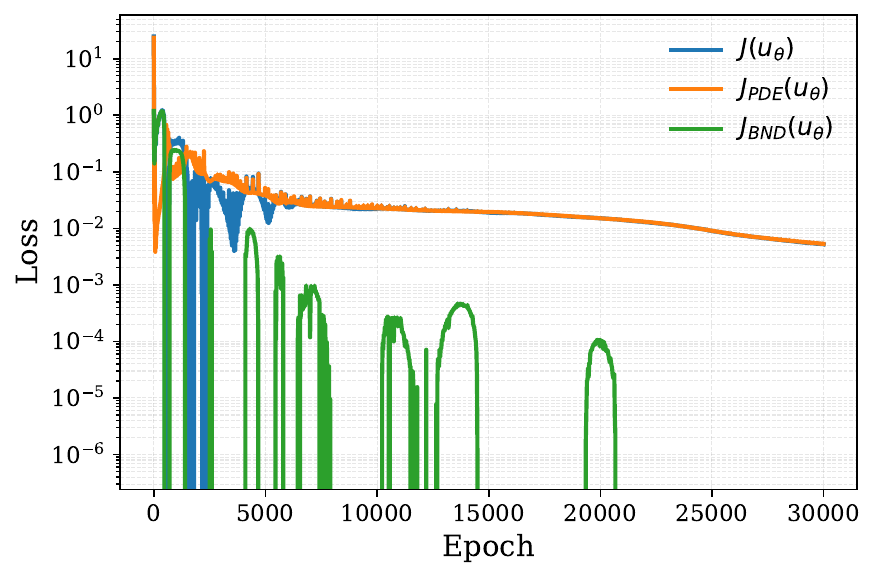}
    \caption{Training loss history.}
    \label{fig:example3:loss}
  \end{subfigure}

  \caption{Summary of results for Example 2. The method accurately resolves beam propagation and shadowing due to a high-absorption obstacle.}
  \label{fig:example3:summary}
\end{figure}

\subsection*{Example 3: Effect of Angular Scattering}

We study the impact of angular redistribution by comparing isotropic
and strongly forward-peaked scattering. Let $\sigma_A = 0.1$ and
$\sigma_S = 9.9$, so that scattering dominates. The source term is
zero, $f \equiv 0$, and the inflow condition is set uniformly as $g
\equiv 1$ on $\Gamma^-$.

We compare two scattering kernels. In the isotropic case, energy is
redistributed uniformly in all directions
\begin{equation}
  K(\vec{\omega}', \vec{\omega}) = \frac{1}{2\pi}.
\end{equation}
In the forward-peaked case, we take
\begin{equation}
  K(\vec{\omega}', \vec{\omega}) = \frac{1}{Z_\epsilon} \exp\left( \frac{\vec{\omega} \cdot \vec{\omega}'}{\epsilon} \right), \qquad \epsilon \ll 1,
\end{equation}
with $Z_\epsilon$ a normalisation constant. This models highly aligned
scattering, where particles preferentially remain close to their
incoming direction.

Figure~\ref{fig:example4:combined} shows the learned solutions in both
cases. Panels (a) and (b) display angular slices of $u_\theta(\vec{x},
\vec{\omega})$ for selected directions $\vec{\omega} = (\cos \theta,
\sin \theta)$, revealing stronger directional persistence in the
forward-peaked case. Panels (c) and (d) show the corresponding scalar
flux $\phi(\vec{x})$, computed by integrating $u_\theta$ over
$\vec{\omega}$. The flux is visibly more focused along characteristics
when the scattering is forward-aligned, as expected from the form of
the kernel.

\begin{figure}[h!]
  \centering
  \begin{subfigure}[t]{0.48\textwidth}
    \centering
    \includegraphics[width=\linewidth]{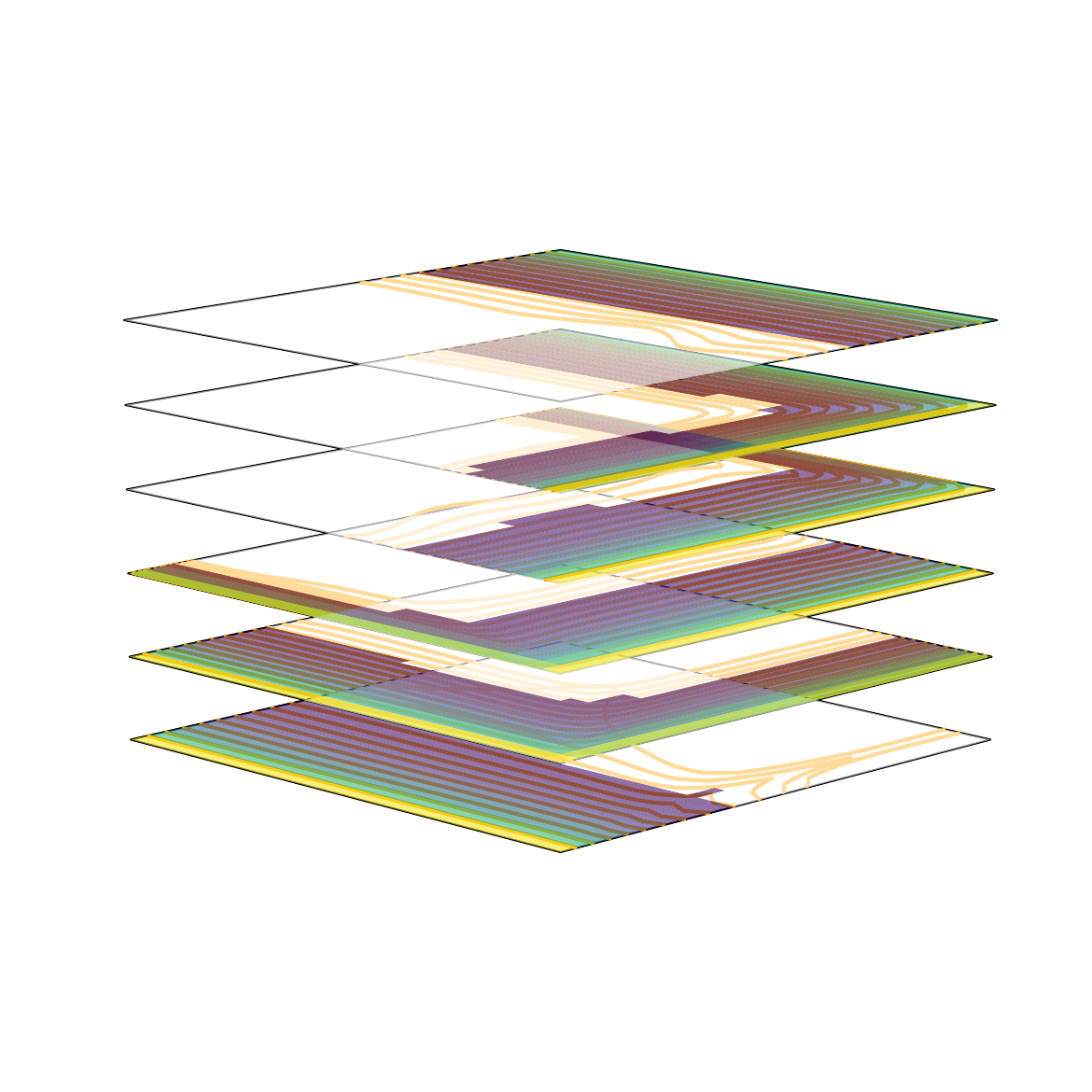}
    \caption{Angular slices for isotropic scattering.}
  \end{subfigure}
  \hfill
  \begin{subfigure}[t]{0.48\textwidth}
    \centering
    \includegraphics[width=\linewidth]{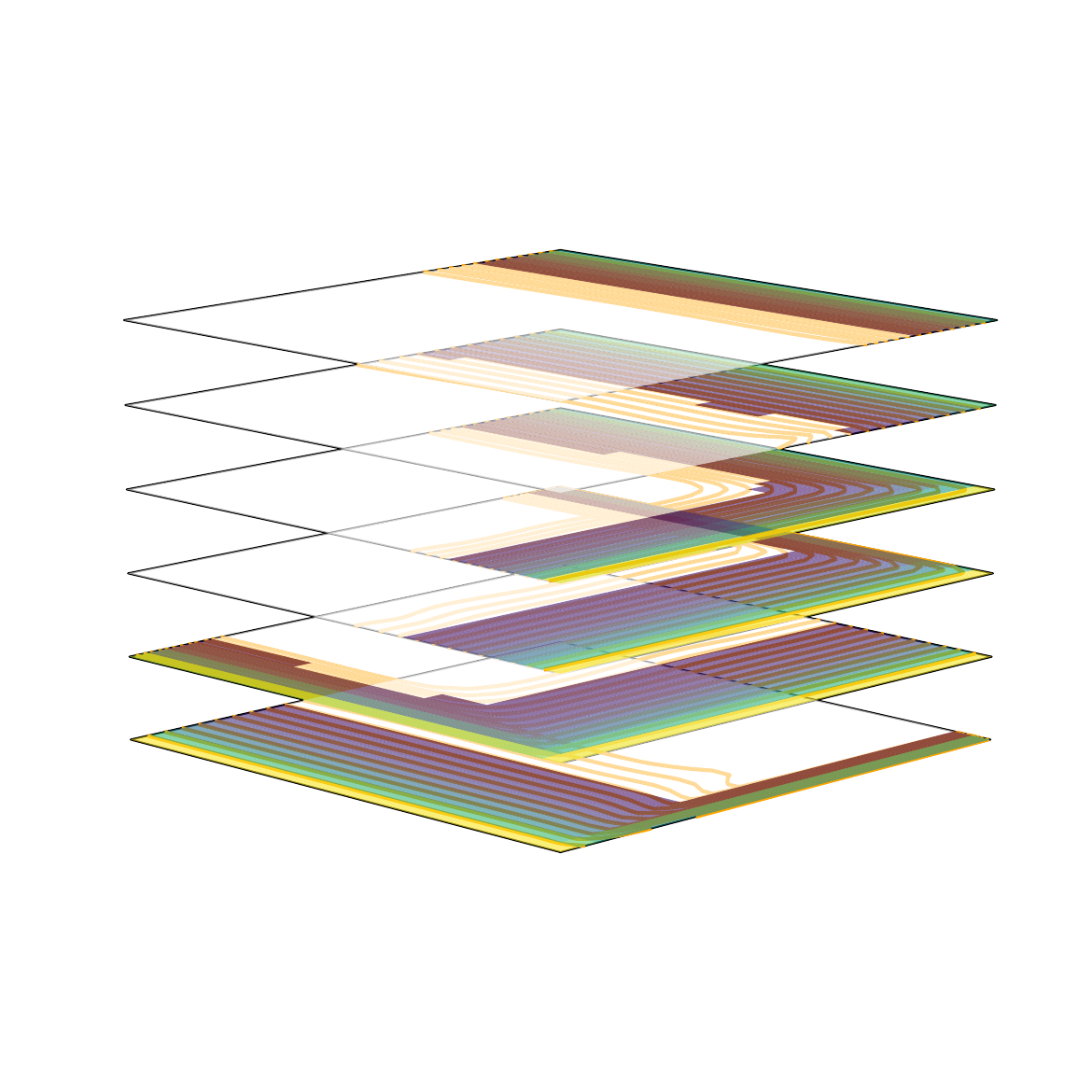}
    \caption{Angular slices for forward-peaked scattering.}
  \end{subfigure}
  \medskip
  \begin{subfigure}[t]{0.48\textwidth}
    \centering
    \includegraphics[width=\linewidth]{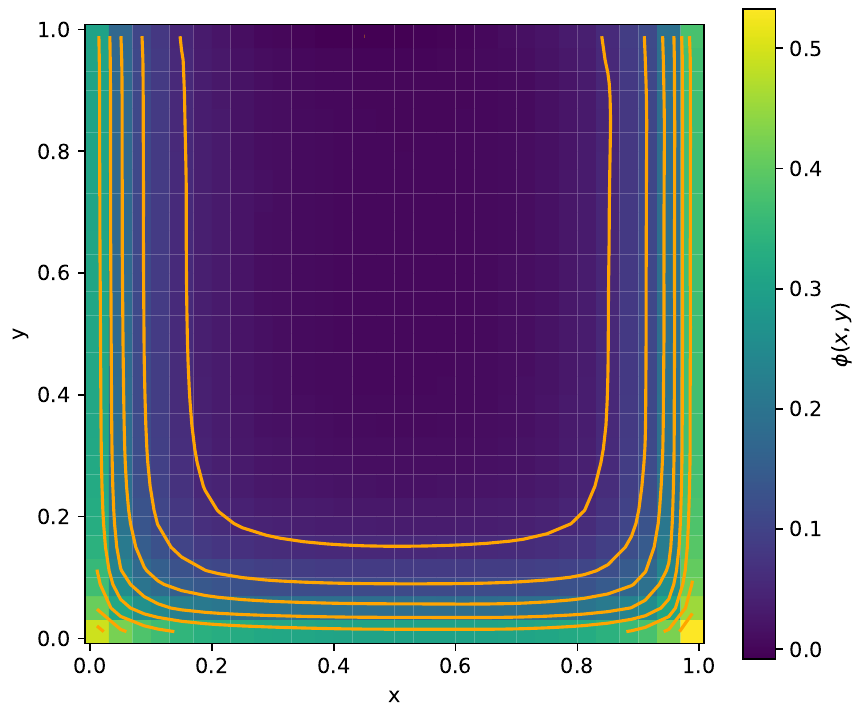}
    \caption{Scalar flux for isotropic scattering.}
  \end{subfigure}
  \hfill
  \begin{subfigure}[t]{0.48\textwidth}
    \centering
    \includegraphics[width=\linewidth]{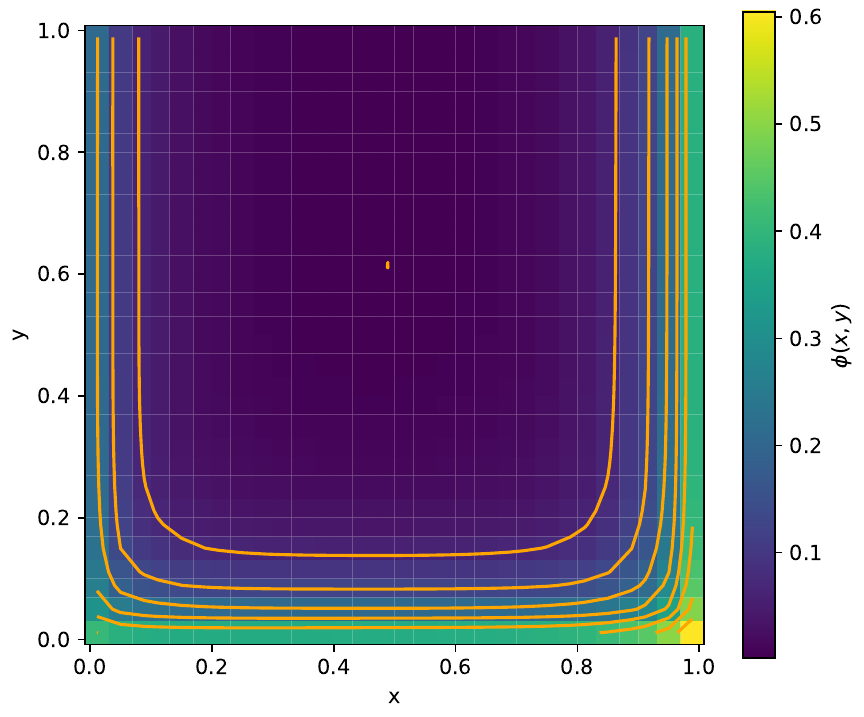}
    \caption{Scalar flux for forward-peaked scattering.}
  \end{subfigure}
  \caption{Summary of results for Example 3. Comparison of learned solutions $u_\theta$ and scalar fluxes under different scattering regimes. Forward-peaked scattering produces greater directional bias and less angular smoothing.}
  \label{fig:example4:combined}
\end{figure}

\subsection*{Example 4: Noisy and Discontinuous Boundary Data}

This test evaluates how the Deep Uzawa scheme handles boundary
conditions that are either discontinuous or contaminated by noise. We
set $f \equiv 0$, $\sigma \equiv 1$, and define the inflow data as
\begin{equation}
  g(\vec{x}, \vec{\omega}) = 
  \begin{cases}
    1 + \delta(\vec{x}, \vec{\omega}), & \text{if } x_1 = 0, \\
    0, & \text{otherwise},
  \end{cases}
\end{equation}
where $\delta(\vec{x}, \vec{\omega})$ is a sample of Gaussian noise
with zero mean and standard deviation $0.05$. This creates a sharp
interface in space combined with angular noise, and provides a
challenging test of the method’s robustness.

Figure~\ref{fig:example5:diagnostics} illustrates the learned solution
and associated quantities. Panel (a) shows angular slices of the
solution $u_\theta(\vec{x}, \vec{\omega})$, evaluated at fixed
$\vec{\omega} = (\cos \theta, \sin \theta)$, confirming that the
method preserves directional structure despite boundary
fluctuations. Panel (b) displays the inflow trace of $u_\theta$, where
the noise is present but partially smoothed by the training
process. Panel (c) shows the scalar flux, where the effect of noise is
spatially damped away from the inflow boundary. Panel (d) plots the
training loss, which exhibits stable convergence despite the
irregularity of the data. Together, these results confirm that the
Uzawa formulation enforces the boundary condition in a weak sense and
allows the network to regularise noise internally while satisfying the
constraint in expectation.

\begin{figure}[h!]
  \centering
  \begin{subfigure}[t]{0.48\textwidth}
    \centering
    \includegraphics[width=\linewidth]{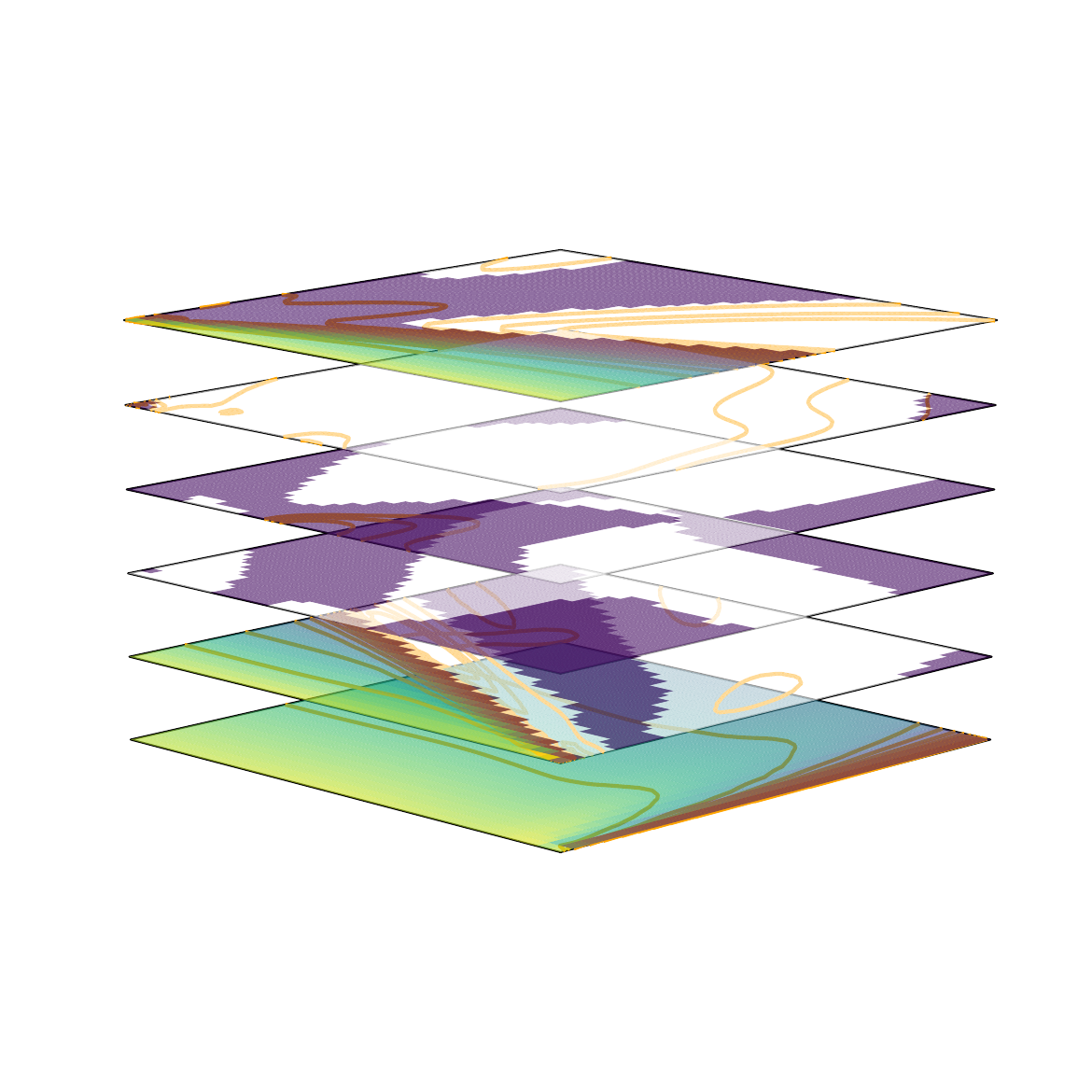}
    \caption{Angular slices of the solution $u_\theta(\vec{x}, \vec{\omega})$.}
  \end{subfigure}
  \hfill
  \begin{subfigure}[t]{0.48\textwidth}
    \centering
    \includegraphics[width=\linewidth]{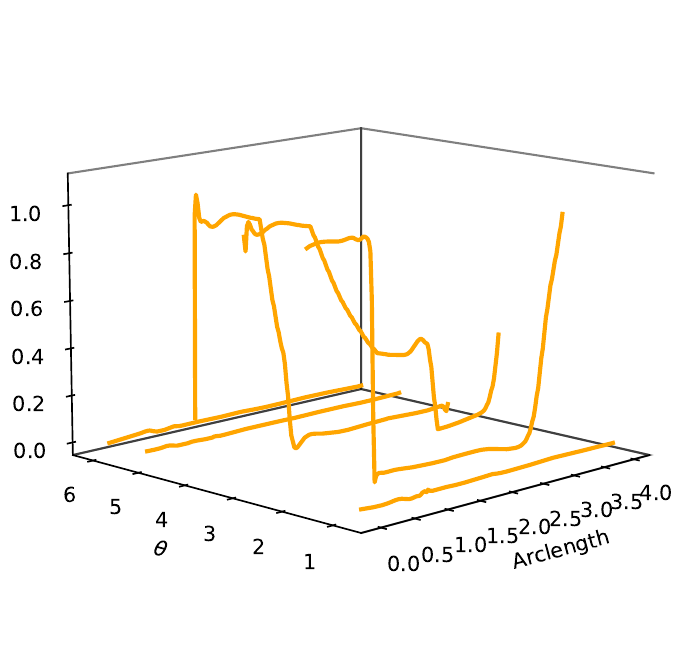}
    \caption{Trained solution on the noisy inflow boundary $\Gamma^-$.}
  \end{subfigure}
  \medskip
  \begin{subfigure}[t]{0.48\textwidth}
    \centering
    \includegraphics[width=\linewidth]{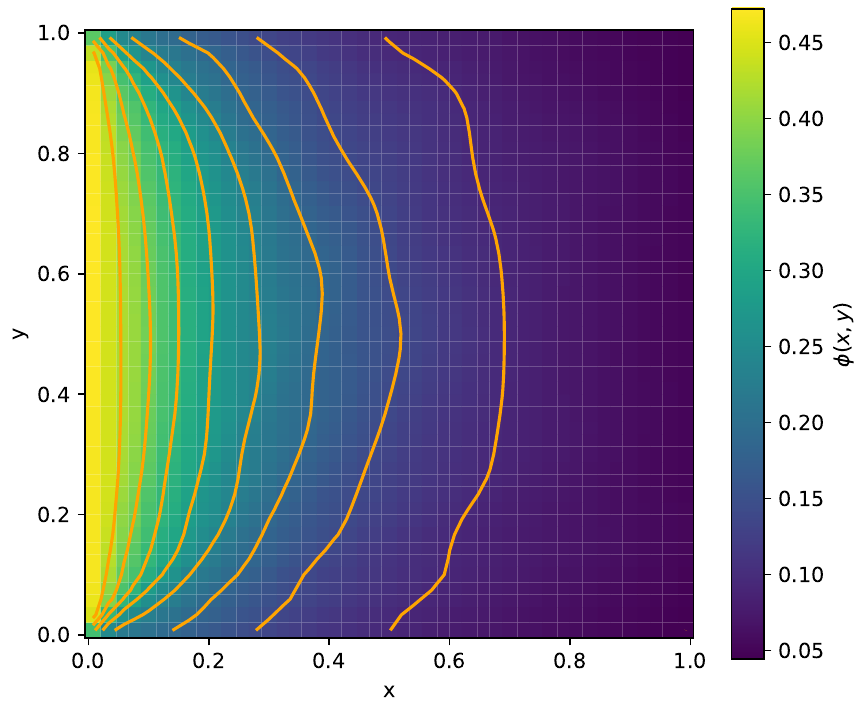}
    \caption{Scalar flux $\phi(\vec{x}) = \int_S u_\theta(\vec{x}, \vec{\omega})\, \mathrm{d}\vec{\omega}$.}
  \end{subfigure}
  \hfill
  \begin{subfigure}[t]{0.48\textwidth}
    \centering
    \includegraphics[width=\linewidth]{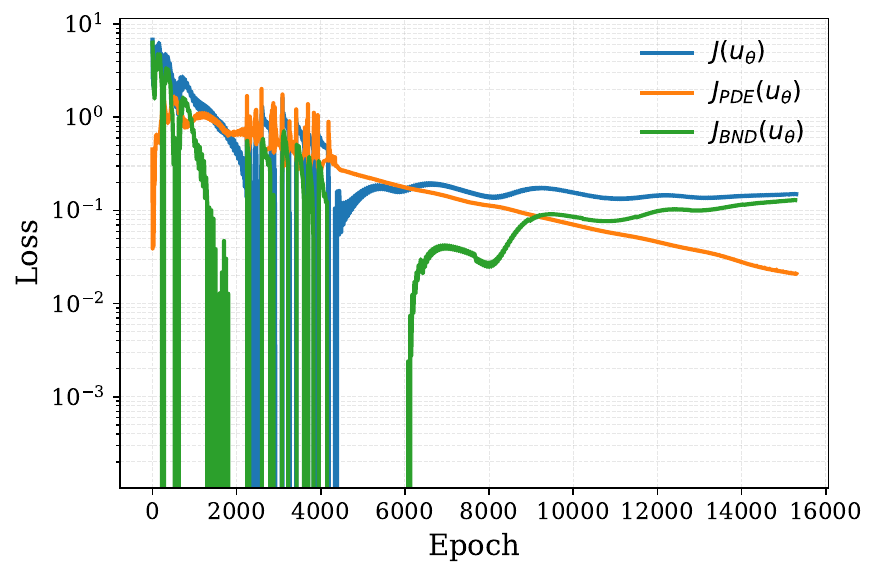}
    \caption{Training loss history.}
  \end{subfigure}
  \caption{Summary of results for Example 4. Diagnostics for noisy
    inflow data. Despite discontinuities and random fluctuations, the
    learned solution remains smooth in the interior, and the boundary
    mismatch is controlled.}
  \label{fig:example5:diagnostics}
\end{figure}

\subsection*{Example 5: Heterogeneous Medium}

To evaluate the ability of the Deep Uzawa scheme to capture spatially heterogeneous coefficients, we define a piecewise constant absorption profile:
\begin{equation}
  \sigma(\vec{x}) = 
  \begin{cases}
    0.1, & x_1 < 0.5, \\
    5, & x_1 \geq 0.5.
  \end{cases}
\end{equation}
We set $f \equiv 0$ and impose uniform inflow data $g \equiv 1$ on
$\Gamma^-$. The goal is to assess whether the network can represent
spatial transitions correctly, in particular the reduced penetration
depth induced by the high absorption on the right side of the domain.

Figure~\ref{fig:example6:diagnostics} displays the resulting network
solution and associated diagnostics. Panel (a) shows angular slices of
the learned solution $u_\theta(\vec{x}, \vec{\omega})$, evaluated for
selected directions $\vec{\omega} = (\cos \theta, \sin \theta)$, which
clearly reflect the spatial variation in absorption. Panel (b)
confirms that the uniform boundary data is captured accurately across
all directions. Panel (c) presents the scalar flux, which exhibits a
sharp decay beyond $x_1 = 0.5$, matching the jump in
absorption. Finally, panel (d) displays the training loss, which
decays steadily, supporting the robustness of the method under
discontinuous coefficients.

\begin{figure}[h!]
  \centering
  \begin{subfigure}[t]{0.48\textwidth}
    \centering
    \includegraphics[width=\linewidth]{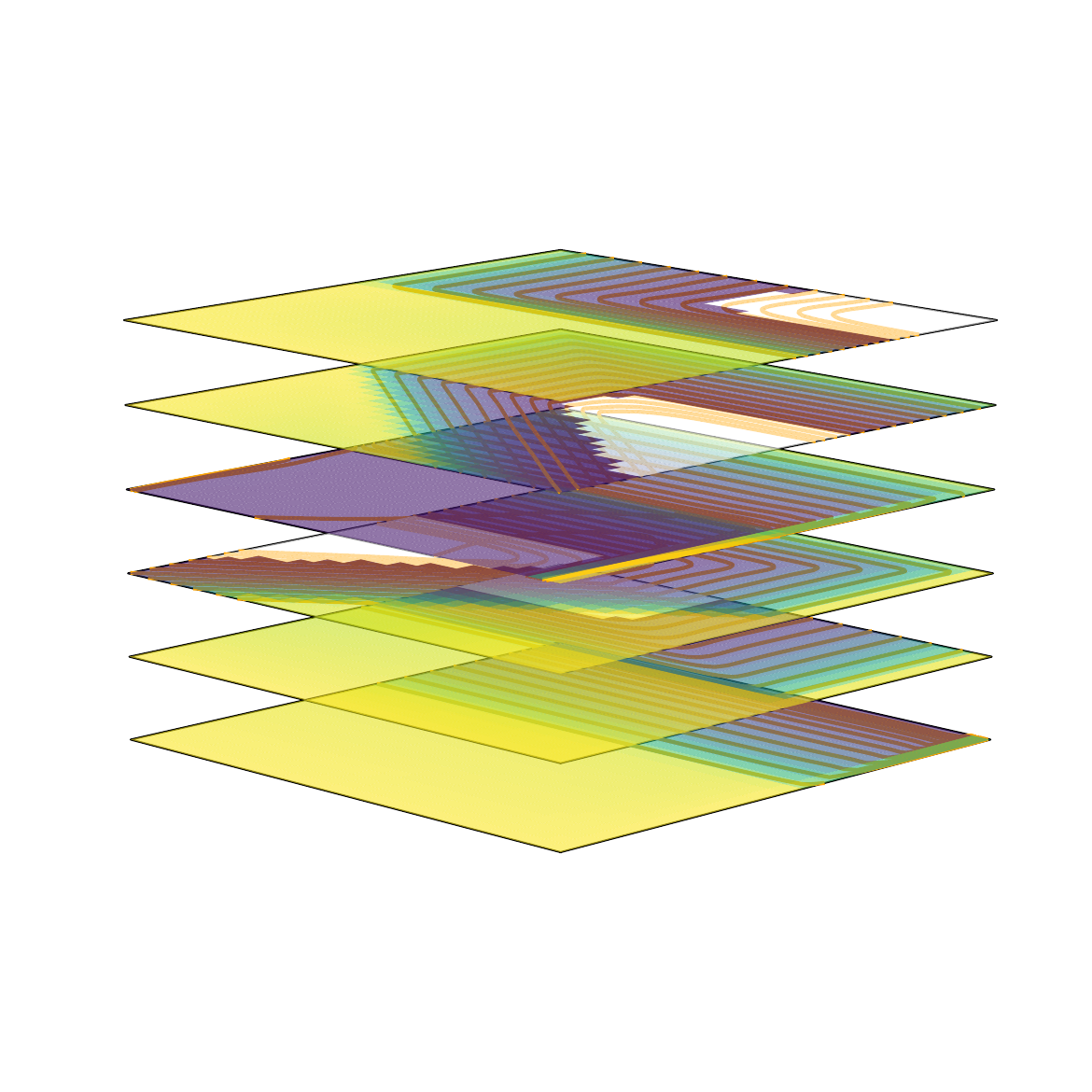}
    \caption{Angular slices of $u_\theta(\vec{x}, \vec{\omega})$.}
  \end{subfigure}
  \hfill
  \begin{subfigure}[t]{0.48\textwidth}
    \centering
    \includegraphics[width=\linewidth]{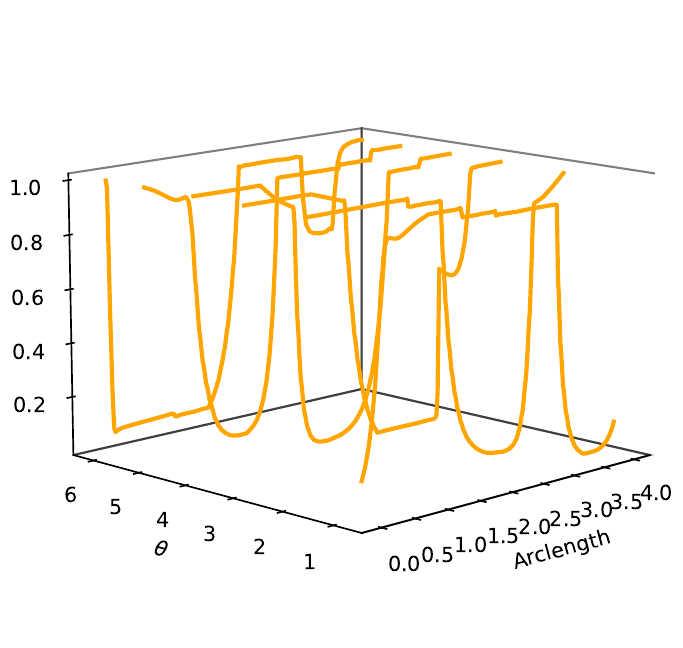}
    \caption{Trained solution on the inflow boundary $\Gamma^-$.}
  \end{subfigure}
  \medskip
  \begin{subfigure}[t]{0.48\textwidth}
    \centering
    \includegraphics[width=\linewidth]{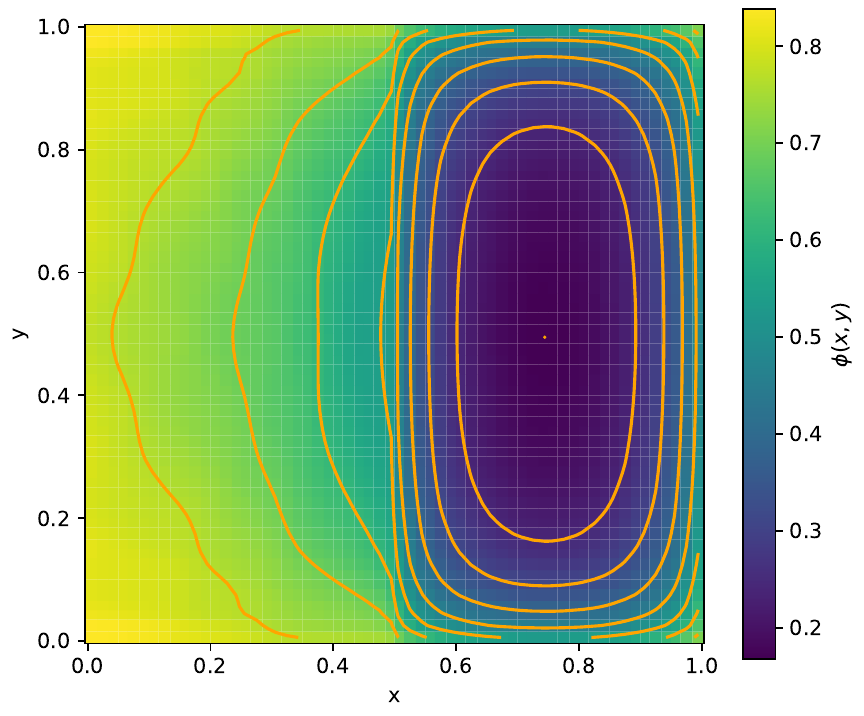}
    \caption{Scalar flux $\phi(\vec{x}) = \int_S u_\theta(\vec{x}, \vec{\omega})\, \mathrm{d}\vec{\omega}$.}
  \end{subfigure}
  \hfill
  \begin{subfigure}[t]{0.48\textwidth}
    \centering
    \includegraphics[width=\linewidth]{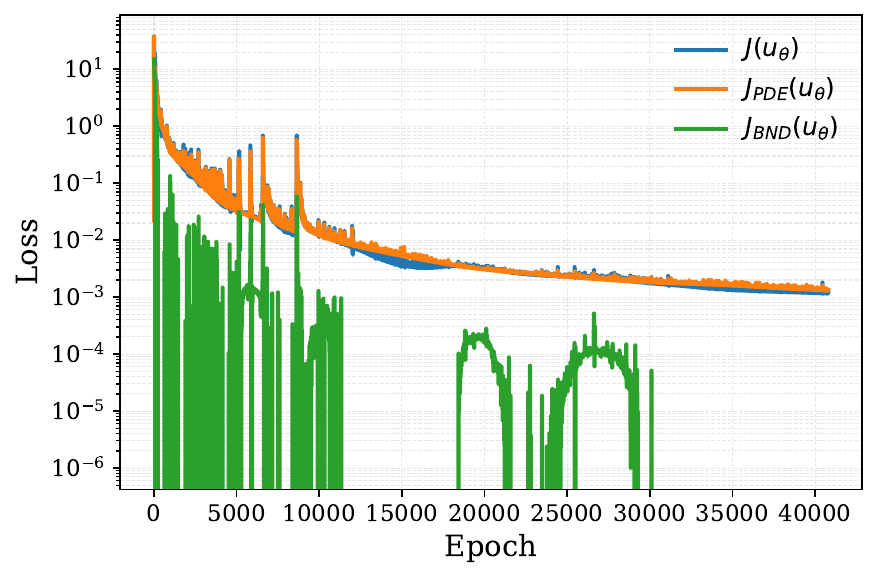}
    \caption{Training loss history.}
  \end{subfigure}
  \caption{Summary of results for Example 5. Diagnostics for a
    heterogeneous medium. The solution exhibits appropriate
    attenuation across the absorption interface without oscillations,
    and the training process remains stable.}
  \label{fig:example6:diagnostics}
\end{figure}

\section{Conclusion}
\label{sec:conc}

We have presented a neural network framework for stationary linear
transport equations with inflow boundary conditions. The method
employs a variational formulation with a Lagrange multiplier on the
inflow boundary, implemented via a Deep Uzawa iteration.  We
established convergence in expectation under mild assumptions,
accounting separately for approximation, optimisation and quadrature
errors. Numerical experiments confirm that the method captures
anisotropic transport, scattering, sharp interfaces and noisy boundary
data. The scheme is mesh-free, stable under parameter variation and
readily extensible to high-dimensional problems.

Future work includes time-dependent extensions, adaptive sampling
strategies and applications to inverse transport and data
assimilation.

\section*{Acknowledgments}

AP and TP were supported by the EPSRC programme grant EP/W026899/1. TP
also received support from the Leverhulme Trust grant RPG-2021-238 and
TP the EPSRC grant EP/X030067/1. The authors also acknowledge support
by the Hellenic Foundation for Research and Innovation (H.F.R.I.)
under the \emph{2nd Call for H.F.R.I.~Research Projects to support
  Post-Doctoral Researchers} (project number: $01247$). All this
support is gratefully acknowledged.

\printbibliography

\end{document}